\documentclass[letter,12pt]{article}
\usepackage{extsizes}
\usepackage{blindtext}
\usepackage[utf8]{inputenc}
\usepackage[square,sort,comma,numbers]{natbib}
\usepackage{amsmath}
\usepackage{amssymb}
\usepackage{amsthm}
\usepackage{pst-node}
\usepackage{blindtext}
\usepackage{geometry}
 \usepackage{float}
 \usepackage{amssymb}
\usepackage{tikz-cd}
\usepackage{amscd}
\usepackage[all,color]{xy}
\usepackage{sseq}
\usepackage{amsfonts}
\usepackage{enumerate}
\usepackage{graphicx}
\usepackage{fancyhdr}
\usepackage{tikz}
\usepackage{graphics}
\usepackage{hyperref}

\theoremstyle{plain} \numberwithin{equation}{section}

\newtheorem{theorem}{Theorem}[section]

\newtheorem{corollary}[theorem]{Corollary}

\newtheorem{lemma}[theorem]{Lemma}

\newtheorem{proposition}[theorem]{Proposition}
\theoremstyle{definition}
\newtheorem{definition}[theorem]{Definition}

\newtheorem{remark}[theorem]{Remark}

\newcommand{\poverline}{\mathfrak{\Bar{p}}}
\newcommand{\rotq}{rot_{\mathbb{Q}}}
\newcommand{\tbq}{tb_{\mathbb{Q}}}

\newcommand{\pa}{\partial}

\def\HFa {\operatorname{\widehat{HF}}}

\newcommand{\s}{$\mathfrak{s}$}
\def \f-{f^{-1}}
\def \fp-{f^{-1}_\partial}

\usepackage{hyperref}

\title{Naturality of Legendrian LOSS invariant under positive contact surgery}
\author{Shunyu Wan}
\date{}
\usepackage{natbib}
\usepackage{graphicx}

\begin{document}

\maketitle

\begin{abstract}

Ozsv\'ath and Stipsicz \citep{OSct} showed that the LOSS invariant is natural under +1 contact surgery. We extend their result and prove the naturality of the LOSS invariant of a Legendrian $L$ under any positive integer contact surgery along a Legendrian $S$ in the complement of $L$. In addition, when $S$ is rationally null-homologous, we also entirely characterize the $Spin^c$ structure in the surgery cobordism that makes the naturality of  contact invariant or LOSS invariant work (without conjugation ambiguity). In particular this implies that contact invariant of the $+n$ contact surgery along a rationally null-homologous Legendrian $S$ depends only on the classical invariants of $S$. The additional generality provided by those results allows us to prove that if two Legendrian knots have different LOSS invariants then after adding the same positive twists to each in a suitable sense, the two new Legendrian knots will also have different LOSS invariants. This leads to new infinite families of examples of Legendrian (or transverse) non-simple knots that are distinguished by their LOSS invariants.

\end{abstract}

\section{Introduction} 

Given a null-homologous oriented Legendrian knot $L$ in a contact 3 manifold $(Y,\xi)$ one can associate the ``LOSS invariant'' $\mathfrak{L}(L)$, and ``LOSS-hat invariant'' $\mathfrak{\widehat{L}}(L)$ to $L$ \citep{LOSS}, which live in the knot Floer groups $HFK^-(-Y,L)$ and $\widehat{HFK}(-Y,L)$  respectively, and are due to Lisca, Ozsv\'ath, Stipsicz and Szab\'o (\citep{OShk}, \citep{Rfk}). In \citep{OSct} Ozsva\'th and Stipsicz study the naturality of LOSS invariant under contact +1 surgery; here we extend that result to contact $+n$ surgery. 

Recall that doing contact surgery on a Legendrian knot in a contact 3-manifold $(Y,\xi)$ gives new contact 3-manifold, but if we are doing contact $+n$ surgery for $n>1$ the resulting contact structure is not unique and we need to make a choice of stabilization \citep{DGS}. Throughout the paper we are choosing the contact structure corresponding to all stabilizations being negative, and denote by $\xi_n^-$ the resulting contact structure (see section 3 for more detail). For the naturality result we prove, we only consider $\xi_n^-$.

\begin{theorem} {\label{th 1.1}}

Let $L,S \in (Y, \xi)$ be two disjoint oriented Legendrian knots in the contact 3-manifold $(Y, \xi)$ with $L$ null-homologous. Let $(Y_n(S),\xi_n^-(S))$ denote the contact 3-manifold we get by performing contact $(+n)$-surgery along $S$, and denote by $L_S$ the oriented Legendrian knot corresponding to L in $(Y_n(S),\xi_n^-(S))$. Moreover suppose that $L_S$ is null-homologous in $Y_n(S)$. Let $W$ be the 2-handle cobordism from $Y$ to $Y_n(S)$ induced by the surgery, and let
\begin{equation}
    F_{S,\mathfrak{s}}: HFK^- (-Y,L) \rightarrow HFK^-(-Y_n(S),L_S)
\end{equation}
be the homomorphism in knot Floer homology induced by $-W$, the cobordism with reversed orientation, for  $\mathfrak{s}$ a $Spin^c$ structure on $-W$. If Y is a rational homology sphere then there is a choice of \s \ for which 

\begin{equation}
    F_{S,\mathfrak{s}}(\mathfrak{L}(Y,\xi,L))=\mathfrak{L}(Y_n(S),\xi_n^-(S),L_S)
\end{equation}
holds. 
A similar identity holds for the Legendrian invariant $\mathfrak{\widehat{L}}$ in $\widehat{HFK}$.  
\end{theorem}

Since the LOSS invariant stays unchanged under negative stabilization \citep{LOSS}, it gives rise to an invariant of transverse knots. If we have a transverse knot $T$ in $(Y,\xi)$ the transverse invariants $\mathfrak{T}$ and $\mathfrak{\widehat{T}}$ are defined to be the LOSS invariants of a Legendrian approximation of $T$ \citep{Eltk}. Thus we obtain a parallel naturality statement for transverse invariants $\mathfrak{T}$ and $\mathfrak{\widehat{T}}$.

\begin{corollary}{\label{cor 1.2}}
   Let $T$ be a null-homologous positively transverse knot and $S$  an oriented Legendrian knot in $(Y, \xi)$ which is disjoint from $T$. Let $(Y_n(S),\xi_n^-(S))$ denote the contact 3-manifold we get by performing contact $(+n)$-surgery along $S$, and we denote $T_S$ the transverse knot corresponding to $T$ in $(Y_n(S),\xi_n^-(S))$. Moreover suppose that $T_S$ is null-homologous in $Y_n(S)$. Let $W$ be the 2-handle cobordism from $Y$ to $Y_n(S)$ induced by the surgery, and let
\begin{equation}
    F_{S,\mathfrak{s}}: HFK^- (-Y,T) \rightarrow HFK^-(-Y_n(S),T_S)
\end{equation}
be the homomorphism in knot Floer homology induced by $-W$, the cobordism with reversed orientation, for  $\mathfrak{s}$ a $Spin^c$ structure on $-W$. If Y is a rational homology sphere then there is a choice of \s \ for which 

\begin{equation}
    F_{S,\mathfrak{s}}(\mathfrak{L}(Y,\xi,T))=\mathfrak{L}(Y_n(S),\xi_n^-(S),T_S)
\end{equation}
holds. 
A similar identity holds for the transverse invariant $\mathfrak{\widehat{T}}$ in $\widehat{HFK}$.  
\end{corollary}

The way of proving Theorem \ref{th 1.1} combines the ideas of \citep{OSct}, \citep{Bc}, \citep{MTn}, and \citep{LSct}, and can be briefly described as follows. We first interpret the contact $+n$ surgery cobordism as a capping off cobordism by viewing it upside down. Then we construct a doubly pointed Heegaard triple  describing the capping off cobordism and the induced map $F_{B,\mathfrak{s}}$ in knot Floer homology where $B$ is the binding component being capped off, and finally we show this map carries the LOSS invariant of $L$ to the LOSS invariant of $L_S$. 

In paricular, Theorem \ref{th 1.1} follows from a naturality property for the LOSS invariant under capping off cobordisms. Recall that an (abtract) open book consists of a pair $(P, \phi)$ where $P$ is a compact oriented 2-manifold with boundary and $\phi$ is a diffeomorphism of $P$ fixing $\partial P$. If a boundary component $B$ of $P$ is chosen, then the capped-off open book $(P', \phi')$ is obtained by attaching a disk to $P$ along $B$ and extending $\phi$ by the identity.

\begin{theorem}\label{capoffthm}
Let $(P_{g,r},\phi)$ be an abstract open book with genus $g$ and $r>1$ binding components. Suppose $T$ and $B$ are distinct binding components; then capping off  $B$ we get a new open book $(P_{g,r-1},\phi')$ which has a binding component $T'$ correspond to $T$. 

Denote by $(M,\xi)$, $(M',\xi')$ the contact 3 manifolds corresponding to those two open books, so that $T$, $T'$ naturally become transverse knots. The capping off cobordism gives rise to a map 

\begin{equation}
    F_{B,\mathfrak{s}}: HFK^- (-M',T') \rightarrow HFK^-(-M,T)
\end{equation}
where $\mathfrak{s}$ is a $Spin^c$ structure on the cobordism W from $-M'$ to $-M$. If $M'$ is a rational homology sphere, and both $T$, $T'$ are null-homologous, then there is a choice of \s \ for which 

\begin{equation}
    F_{B,\mathfrak{s}}(\mathfrak{T}(M',\xi',T'))=\mathfrak{T}(M,\xi,T)
\end{equation}
holds. A similar
identity holds for the transverse invariant $\mathfrak{\widehat{T}}$ in $\widehat{HFK}$. 
\end{theorem} 

Combining the proof of Theorem \ref{th 1.1} in this paper and the proof of Theorem 1.1 in \citep{MTn}, it's easy to see that $Spin^c$ structures on the cobordism that makes the naturality of contact invariant and LOSS invariant work are the same $Spin^c$ structure, and we can say more about this $Spin^c$ structure $\mathfrak{s}$.

\begin{proposition}
{\label{prop 1.4}} 
Assume in the situation of Theorem \ref{th 1.1} that $S$ is null-homologous and both $Y$ and $Y_n(S)$ are rational homology sphere. Then the $\mathfrak{s}$ in Theorem \ref{th 1.1} and Corollary \ref{cor 1.2}, as well as in \cite[Theorem 1.1 (for integer surgery)]{MTn} has the property that $$ \langle c_1(\mathfrak{s}),[\Tilde{F}] \rangle = 
      rot(S)+n-1 $$
 where $F$ is a Seifert surface for S and $\Tilde{F}$ is obtained by attaching the core of the $2$-handle to $F$ in $W$.
\end{proposition}

The above proposition is a direct consequence of Theorem \ref{thm: spin^c for RHS} which is an analogous statement but for contact $+n$ surgery on a rationally null-homologous Legendrian $S$.

\begin{remark}
    In \citep{MTn} Mark-Tosun prove the above proposition (for the contact invariant) but with a \textbf{sign} ambiguity (such sign ambiguity also shows up in \cite{DingLiWuContact+1onrationalsphere}). Here we resolve this ambiguity. In other words we can characterize this $Spin^c$ structure not just up to conjugation, and the proof is entirely different. 
\end{remark}

As a direct consequence of Theorem \ref{thm: spin^c for RHS} we have the following corollary which extends a result of Golla \citep[Proposition 6.10]{MarcoOSinvariantsofcontactsurgeries} from $(S^3,\xi_{std})$ to all contact rational homology sphere (Golla does not require the smooth coefficient to be non-zero but we do). 

\begin{corollary}
Let $(Y,\xi)$ be a contact rational homology sphere, $S$ be a rationally null-homologous Legendrian knot, and $L$ be a null-homologous Legendrian knot. If $Y_n(S)$ is again a rational homology sphere, then the contact invariant $c(\xi^-_n(S))$ as well as the LOSS invariant $\mathfrak{L}(Y_n(S),\xi^-_{n}(S),L_S)$ are independent of the Legendrian isotopy class of $S$, when the classical invariants are fixed.
\end{corollary}

\begin{remark}
 The LOSS invariants are actually only well defined up to sign, and up to the action of the mapping class group  on $(Y,L)$ \citep{OSct} (that is, the group of isotopy classes of diffeomorphisms of $Y$ fixing $L$). We denote by $[\mathfrak{L}]\in  HFK^- (-Y,L)/\pm MCG(Y,L)$  the image of $\mathfrak{L}$ when we quotient out these actions, and similarly for the other types of LOSS invariants.
\end{remark}

Using the Theorem \ref{th 1.1} and  Proposition \ref{prop 1.4} one can produce infinite families of smooth knots that have distinct Legendrian (resp. transverse) representatives with same Thurston-Bennequin and rotation numbers (resp. self linking number).  More specifically starting with two Legendrian representatives of knot $K$ with different $[\mathfrak{L}]$ or $[\mathfrak{\widehat{L}}]$, it is always possible to produce two Legendrian representatives of a new knot $K'$ that also have different $[\mathfrak{L}]$ or $[\mathfrak{\widehat{L}}]$, essentially by adding positive twists to parallel strands in $K$. The procedure can be described more precisely as follows. 

Let $(Y,\xi)$ be a contact 3-manifold, and consider a triple $(L, \sigma_n, B)$ where $L$ is a Legendrian knot in $(Y,\xi)$, $\sigma_n=$ \{$e_i|e_i$ is  an  oriented  Legendrian arc  of  $L$  for  $i=1,2,...,n$\}, and $B$ is a Darboux ball. We say this is a \textbf{compatible triple} if the following hold:

\begin{enumerate}
    \item $B$ only intersects $L$ at $\sigma_n$. 
    \item Inside the Darboux ball the front projections of the arcs $e_i$ are horizontal, parallel and have the same orientation for all $i=1,2,...,n$. In other words $\sigma_n$ is a collection of $n$ Legendrian push offs of one oriented horizontal arc.
\end{enumerate}

Given a compatible triple $(L, \sigma_n, B)$ we can construct a new oriented Legendrian knot {\boldmath $L_\sigma$} by adding a full non-zigzagged positive twist to the front projection of $\sigma_n$ in $B$ (See figure \ref{New L} for example when $n=2$).

\begin{figure}
    \centering
    \includegraphics[width=\textwidth]{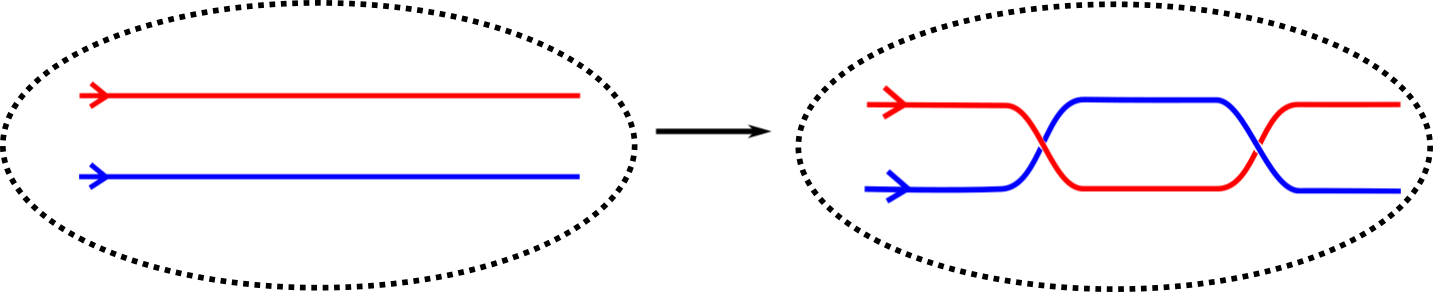}
    \caption{Example when there are two parallel arcs (the blue and red arcs are $e_1$ and $e_2$, and the dotted circle represents a Darboux ball). On the left is part of $L$ inside a standard Darboux ball. After doing the twist we get the right diagram which is still inside the Darboux ball and is part of the new knot $L_\sigma$}
    \label{New L}
\end{figure} 

Note that since all arcs $e_i$ are horizontal, parallel and oriented in the same direction there is no ambiguity of the new Legendrian knot $L_\sigma$ once given a compatible $(L, \sigma_n, B)$. Now we are able to state the theorem. 

\begin{theorem}
{\label{th 1.5}}
    In the above setting let $(L, \sigma_n, B)$ and $(L', \sigma_n', B')$ be two compatible triples in $(Y,\xi)$. Assume  
    
    \begin{enumerate}
        \item [$\bullet$] $L$ and $L'$ are smoothly ambiently isotopic
        \item [$\bullet$] The isotopy sends $B$ contactmorphically to $B'$, and $e_i$ to $e_i'$.
    \end{enumerate} 
    Then $L_\sigma$, $L'_{\sigma'}$ are smoothly isotopic.  Moreover if Y is a rational homology sphere and $L$ is null-homologous, and if $L$ and $L'$ have different $[\mathfrak{L}]$ or $[\mathfrak{\widehat{L}}]$, then so do $L_\sigma$ and $L'_{\sigma'}$.
    \end{theorem}

As an example (application) of Theorem 1.6 we will see the following corollary.
\begin{corollary}
{\label{cor 1.6}}
    In standard tight $S^3$, the mirror of the knot   $9_7$ in Rolfsen’s table is neither Legendrian simple nor transversely simple. 
\end{corollary}

We will see  in section \ref{section 5.2}  there are many more non-simple knot examples  that we can derive from  Theorem \ref{th 1.5} (see Theorem \ref{th 5.2}, Corollary \ref{th 5.3}, and the discussion about figure \ref{(2,3)(2,3)}).

The paper is organized as follows. In section 2 we will briefly review some basic preliminaries of knot Floer theory, the LOSS invariant, and maps induced by surgery. In section 3 we will talk about  contact $+n$ surgery and capping off. Then in section 4 we will prove  Theorem \ref{th 1.1}. In section 5 we will prove Theorem \ref{th 1.5}, Corollary \ref{cor 1.6}, and see more examples of non-simple knots. Finally in section 6 we will prove the $Spin^c$ structure formula for contact surgery on rationally null-homologous knot (c.f. Theorem \ref{thm: spin^c for RHS}) that implies Proposition \ref{prop 1.4}.\\

\textbf{Acknowledgements.} The author would like to thank his advisor Tom Mark for his patience and numerous helpful guidance and suggestions. The author would also like to thank B\"ulent Tosun for early stage discussions, and John Etnyre and Tye Lidman for useful conversations about the applications. The author was supported in part by grants from the NSF (RTG grant DMS-1839968) and the Simons Foundation (grants 523795 and 961391 to Thomas Mark).

\section{Knot Floer Preliminaries}

\subsection{Knot Floer Homology}
We will use the same notation and construction as in \citep{OShk}. A doubly pointed Heegaard diagram $(\Sigma,\alpha, \beta, w,z)$ consists of the following information.  $\Sigma$ is an oriented genus $g$ surface, $\alpha=\{\alpha_1,...,\alpha_g\}$ is a $g$-tuple of disjoint homologically linearly independent circles on  $\Sigma$, $\beta=\{\beta_1,...\beta_g\}$ is another $g$-tuple of circles on $\Sigma$ similar to $\alpha$, and $z,w$ are two points on the complement of the $\alpha$ and $\beta$ curves. Such a diagram gives rise to a 3-manifold $Y$ in a standard way, by thinking of the $\alpha$ and $\beta$ circles as determining (the compressing disks in) handlebodies $H_\alpha$ and $H_\beta$ with $\partial H_\alpha = -\partial H_\beta = \Sigma$, and setting $Y = H_\alpha\cup_\Sigma H_\beta$.

Given an oriented null-homologous knot $K$ in some three manifold $Y$, one can construct a doubly pointed Heegaard diagram $(\Sigma,\alpha, \beta,w,z)$ describing $(Y,K)$ in the following sense. $(\Sigma, \alpha,\beta)$ is a Heegaard diagram for $Y$, and if we connect $z$ to $w$ by an embedded arc missing the $\alpha$ circles and pushed  a little bit into the $\alpha$ handlebody, and connect $w$ to $z$ by another embedded arc missing the $\beta$ circles and pushed into the $\beta$ handlebody, then the closed curve given by the union of those two arcs is exactly the knot $K$ in Y.  We say the doubly pointed Heegaard diagram $(\Sigma,\alpha, \beta,w,z)$ is compatible with $(Y,K)$. 

We can further associate a chain complex $CFK^-$ to a doubly pointed Heegaard diagram $(\Sigma,\alpha, \beta,w,z)$. Assume the $\alpha$ and $\beta$ curves intersect transversely, and consider the two tori $$\mathbb{T}_\alpha=\alpha_1 \times \alpha_2 \times ... \times \alpha_g, \quad \mathbb{T}_\beta=\beta_1 \times \beta_2 \times ... \times \beta_g$$ in the $g^{th}$ symmetric power $Sym^g(\Sigma)$. The chain complex $CFK^-$ is the free $\mathbb{F}[U]$-module generated by the intersection points of $\mathbb{T}_\alpha \cap \mathbb{T}_\beta$. The differential $\pa^-$ is defined as following: 

$$\pa^-{\textbf{x}}=\sum_{\{\textbf{y}\in \mathbb{T}_\alpha \cap \mathbb{T}_\beta\}} \sum_{\{\phi \in \pi_2(\textbf{x},\textbf{y}),\mu(\phi)=1,n_z(\phi)=0\}}\# \widehat{\mathfrak{M}}(\phi) \cdot U^{n_w(\phi)}\cdot \textbf{y}$$ 
where $\pi_2(\textbf{x},\textbf{y})$ is the set of homotopy class of disk connecting $\textbf{x}$ to $\textbf{y}$; $\mu(\phi)$ is the expected dimension of the moduli space $\mathfrak{M}(\phi)$ of holomorphic disks in the homotopy class $\phi$; $n_z(\phi)$ and $n_w(\phi)$ are the algebraic intersection numbers between $\phi$ and $\{z\}\times Sym^{g-1}(\Sigma)$ and $\{w\}\times Sym^{g-1}(\Sigma)$, respectively. The knot Floer homology groups $HFK^-$ and $\widehat{HFK}$ are the homology of the complexes $CFK^-$ and $\widehat{CFK}$ respectively, where $\widehat{CFK}$ is the same as $CFK^-$ except specializing $U=0$ throughout. It is shown in \citep{OShk} that under suitable admissiblity conditions for the Heegaard diagram \citep{OSht}, these homology groups are invariants of the smooth knot type $K$ in $Y$.

When the knot is null-homologous these groups are bi-graded and can be decomposed as follows: $${HFK^-(Y,K)=\bigoplus_{d\in \mathbb{Q},\mathfrak{s}\in Spin^c(Y,K)} HFK_d^-(Y,K,\mathfrak{s}),}$$
where $d$ is the Maslov grading and $\mathfrak{s}$, which run through relative $Spin^c$ structures, is the Alexander grading. 

\subsection{LOSS invariant for Legendrian knots}

In this section we review the construction for LOSS invariant \citep{LOSS}. As we describe below, given an oriented null-homologous Legendrian knot $L$ in some contact three manifold $(Y,\xi)$, one can find a doubly pointed Heegaard diagram $(\Sigma, \beta, \alpha,w,z)$ that is compatible with $(-Y,L)$ and associate a cycle in $CFK^-$, giving rise to an element in $HFK^-(-Y,L)$. 

To begin, we can find an open book $(P,\phi)$ that is compatible with $(Y,\xi)$ and contains $L$ as a homologically nontrivial curve on the page \citep[Proposition 2.4]{LOSS}. In particular, we assume we have a fixed homeomorphism between $Y$ and the relative mapping torus of $(P, \phi)$:
\[
Y \cong \frac{P\times [0,1]}{(x,1)\sim (\phi(x), 0)} \,\cup \partial P \times D^2
\] 

where we will blur the distinction between a boundary circle $C$ of $P$ and the core of the corresponding solid torus $C\times \{0\} \subset C\times D^2$. The fibers of the naturally induced map $Y\to S^1$, defined away from the binding circles $C \times \{0\}$, will be denoted by $P_t$ for $t\in S^1$.

Then we choose a family of properly embedded arcs $\bf\{a_i\}$ as basis for $P$, meaning that if we cut $P$ along $\bf\{a_i\}$ we get a disk. We can choose the basis $\bf\{a_i\}$ such that $L$, considered as lying on $P = P_{+1}$, only intersects $a_1$ transversely at one point and does not intersect with other $a_i$ for $i\neq 1$. We now construct a doubly pointed Heegaard diagram $(\Sigma, \beta, \alpha,w,z)$ that is compatible with $(-Y,L)$ using $(P,\phi,\bf\{a_i\})$.   

We first form the Heegaard surface $\Sigma$
as the union of two pages, $P_{+1} \cup -P_{-1}$. For the $\alpha$ and $\beta$ curves we start with the basis $\bf\{a_i\}$ lying on $P_{+1}$, and let $b_i$ be a push off of $a_i$ for all $i$ on $P_{+1}$, such that $a_i$ and $b_i$ intersect transversely at one point on $P_{+1}$ for each $i$. In particular the boundary points of $b_i$ are obtained from those of $a_i$ by pushing along $\partial P_{+1}$ in the direction determined by the orientation. Now we let $\alpha_i=a_i \cup \overline{a_i}$ and $\beta_i=b_i\cup \overline{\phi(b_i)}$ for all $i$, where $\overline{a_i}$ is the image of $a_i$ under the identity
map on the opposite page $-P_{-1}$ and $\overline{\phi(b_i)}$ is the image of $b_i$  under the monodromy map $\phi$ on $-P_{-1}$. Finally we place the base points $w,z$ on $P_{+1}$ such that $z$ is ``outside'' the thin strips between $a_i$ and $b_i$ for all $i$, and $w$ is in between $a_1$ and $b_1$. Note that there are two possibilities for the placement of $w$; we choose the one compatible with the orientation of $L$. Let $ \bf{x}$ $=(x_1,x_2,...,x_g)\in \mathbb{T_\alpha} \cap\mathbb{T_\beta}$ where $x_i=a_i\cap b_i$. Now we change the orientation of $Y$ and consider the diagram $(\Sigma,\beta,\alpha,w,z)$, which is compatible with $(-Y,L)$. We view $\bf{x}$ as an element in $CFK^-(-Y,L)$. It was shown in \citep{LOSS} that $\bf{x}$ is a cycle, and the homology class of $\bf{x}$, written  $\mathfrak{L}(L)\in HFK^-(-Y,L)$, is an invariant of the oriented Legendrian knot $L$ with values in the graded module $HFK^-(-Y, L)$, modulo its graded automorphisms. The construction  for $\widehat{\mathfrak{L}}(L) \in \widehat{HFK}(-Y,L)$ is the same, simply considering $\bf{x}$ as a cycle in $\widehat{CFK}(-Y,L)$.

Here are some properties of the LOSS invariant that we will use.

\begin{theorem}[\citep{LOSS}] \label{2.1}
Suppose that $L$ is an oriented Legendrian knot and denote the negative and  positive stabilizations of $L$ as $L^-$ and $L^+$. Then,
$\mathfrak{L}(L^-) = \mathfrak{L}(L)$ and $\mathfrak{L}(L
^+) = U \cdot \mathfrak{L}(L)$. Similarly $\mathfrak{\widehat{L}}(L^-) = \mathfrak{\widehat{L}}(L)$ and $\mathfrak{\widehat{L}}(L
^+) = U \cdot \mathfrak{\widehat{L}}(L)=0$.
\end{theorem}

In particular, since both $\mathfrak{L}$ and $\mathfrak{\widehat{L}}$ are unchanged under negative stabilization they are also invariants for transverse knots. Thus for a transverse knot $T$ we define the LOSS invariant $\mathfrak{T}, \mathfrak{\widehat{T}}$ of $T$ to be the LOSS invariant $\mathfrak{L}, \mathfrak{\widehat{L}}$ of a Legendrian approximation of $T$. 

\subsection{Maps induced by surgery}

Let $Y$ be a 3-manifold and $K$ be a framed knot in $Y$ with framing $f$, and denote $Y_f(K)$ to be the 3 manifold obtained from $Y$ by surgery along $K$ with framing $f$. Then there exists a Heegaard triple  $(\Sigma, \alpha, \beta, \gamma, z)$ that is ``compatible'' with (or ``subordinate'' to) the cobordism induced by surgery, in particular $(\Sigma, \alpha, \beta)$ describes $Y$, $(\Sigma, \alpha, \gamma)$ describes $Y_f$, and $(\Sigma, \beta,\gamma)$ is a Heegaard diagram for a connected sum of copies of $S^1 \times S^2$, and furthermore we can explicitly relate the framed knot $(K,f)$ to the Heegaard triple (see \citep{OShi} Section 4 for details). This Heegaard triple gives a well defined map from $HF^-(Y)$ to $HF^-(Y_f(k))$, and a similar construction works for knot Floer homology \citep{OShk}, as we now outline. Assume $L$ is a homologically trivial knot and assume the induced knot $L'$  in $Y_f(K)$ is also homologically trivial. Then there exists a doubly pointed Heegaard triple $(\Sigma, \alpha, \beta, \gamma,w, z)$ describing the surgery cobordism and giving rise to a map 

\begin{equation}
    F_{K(f),\mathfrak{s}}: HFK^-(Y,L) \rightarrow HFK^-(Y_f(K),L')
\end{equation} which is induced by a chain map 

\begin{equation}
    f_{K(f),\mathfrak{s}}: CFK^- (Y,L) \rightarrow CFK^-(Y_f(K),L').
\end{equation} The latter is defined for a compatible doubly pointed Heegaard triple $(\Sigma, \alpha, \beta, \gamma,w, z)$ by the formula 
$$f_{K(f),\mathfrak{s}}(\textbf{x})=\sum_{\{\textbf{y}\in \mathbb{T}_\alpha \cap \mathbb{T}_\gamma\}} \sum_{\psi} \# {\mathfrak{M}}(\psi)\cdot U^{n_w(\phi)} \cdot \textbf{y}$$ 
where $\mathfrak{s}$ is a $Spin^c$ structure on the cobordism. The inner sum is over homotopy classes $\psi\in\pi_2(\textbf{x},\Theta_{\beta,\gamma},\textbf{y})$ of Whitney triangles connecting $\textbf{x}$, $\textbf{y}$, and a representative $\Theta_{\beta,\gamma}$  of the top dimensional class in $HFK^-(\Sigma, \beta, \gamma,w, z)$, and satisfying $s_w(\psi) = \mathfrak{s}$, $n_z(\psi) = 0$, and $\mu(\psi) = 0$, where the latter is the expected dimension of the moduli space $\mathfrak{M}(\psi)$ of holomorphic triangles in homotopy class $\psi$. 

Recall that a Whitney triangle connecting $\textbf{x}\in \mathbb{T}_\alpha \cap \mathbb{T}_\beta$, $\textbf{r}\in \mathbb{T}_\beta \cap \mathbb{T}_\gamma$, and $\textbf{y}\in \mathbb{T}_\alpha \cap \mathbb{T}_\gamma$ is a map $$u: \Delta \rightarrow Sym^n(\Sigma)$$ where $\Delta_i$ is an oriented 2-simplex with vertices $v_\alpha$, $v_\beta$, and $v_\gamma$ labeled clockwise, and $e_\alpha$, $e_\beta$, and $e_\gamma$ are the edges opposite to $v_\alpha$, $v_\beta$, and $v_\gamma$ respectively. Moreover we want the boundary conditions that $u(v_\alpha) = \textbf{r}$, $u(v_\beta) = \textbf{y}$ and $u(v_\gamma) = \textbf{x}$, and $u(e_\alpha) \subset \mathbb{T}_\alpha$,
$u(e_\beta) \subset \mathbb{T}_\beta$ and $u(e_\gamma) \subset \mathbb{T}_\gamma$. In particular if we start at any vertex of $\Delta$ and go clockwise we should encounter the $\alpha$, $\beta$, and $\gamma$ curves in cyclic order. See Figure \ref{Whitney triangle} for a schematic picture. 

\begin{figure}[H]
    \centering
    \includegraphics[width=50mm,scale=0.5]{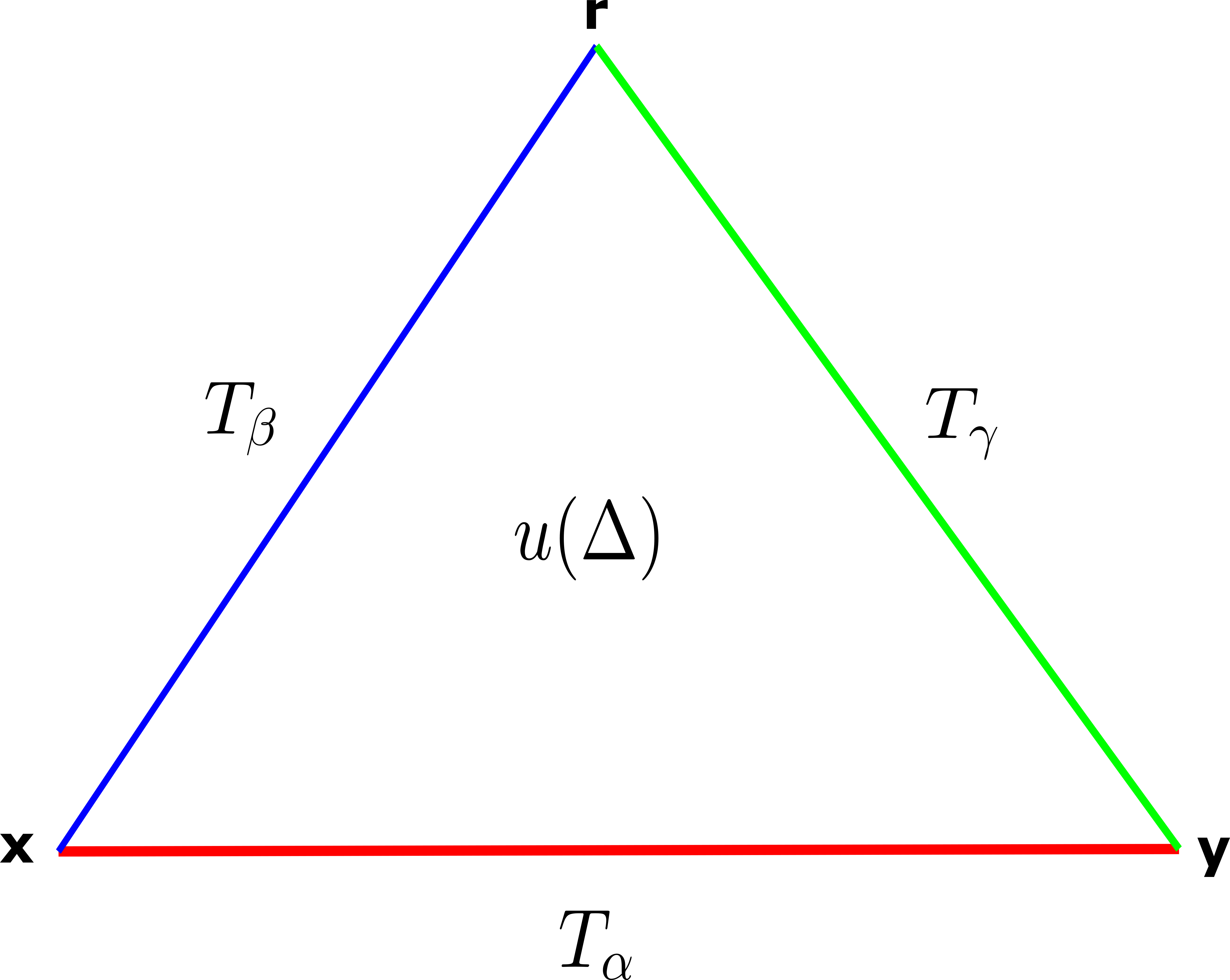}
    \caption{Schematic Whitney triangle for $(\Sigma, \alpha,\beta, \gamma)$
    }
    \label{Whitney triangle}
\end{figure}

\section{Contact surgery and Capping off cobordism} \label{Contact surgery and Capping off cobordism}

In this section we review some basic properties of contact surgery and how it is related to capping off cobordism. 

Given a Legendrian knot $L$ in a contact 3 manifold $(Y,\xi)$ there is a canonical contact framing defined by a vector field along $L$ that is transverse to the 2-plane field $\xi$. In \citep{DGst} Ding and Geiges describe a notion of contact $r$-surgery on $L$ which, for a choice of rational number $r$, gives rise to another contact 3 manifold $(Y_r(L),\xi_r(L))$ 
whose underlying 3-manifold $Y_r(L)$ is obtained by smooth surgery along $L$ with coefficient $r$ relative to the contact framing. (Note that in general there are choices required to completely determine the resulting contact structure for general contact rational $r$-surgery.) In \citep{DGS} Ding, Geiges and Stipsicz prove the following theorem. 

\begin{theorem}
    Every (closed, orientable) contact 3 manifold $(Y,\xi)$ can
be obtained via contact $\pm 1$-surgery on a Legendrian link in $(S^3,\xi_{std})$.
\end{theorem}

Moreover in \citep
{DGS} they describe an algorithm to transform positive rational $r$-surgery on a Legendrian knot $K$ to a sequence of $\pm 1$-surgeries on some (stabilizations of) push offs of $K$ as follows.

\begin{theorem} [DGS algorithm \citep
{DGS}] {\label{3.2}}
Given a Legendrian knot $L$ in $(Y,\xi)$. Let $0<\frac{x}{y}=r\in \mathbb{Q}$ be a contact surgery coefficient. Let $c\in \mathbb{Z}$ be the minimal positive integer such that $\frac{x}{y-cx}<0$, with the continued fraction \begin{equation}
    \frac{x}{y-cx}=[a_1,a_2,...a_m]=a_1-\frac{1}{a_2-\frac{1}{...-\frac{1}{a_m}}}
\end{equation}
where each $a_i\leq -2$. Then any contact $\frac{x}{y}$ surgery on $L$ can be described as contact surgery along a link $(L_0^1 \cup L_0^2 \cup ... \cup L_0^c)\cup L_1 \cup...\cup L_m$, where 
\begin{itemize}
     \item $L_0^1,...,L_0^c$ are Legendrian push offs of $L$.
    \item $L_1$ is obtained from a Legendrian push off of $L_0^c$ by stabilizing $|a_1+1|$ times.
    \item $L_i$ is obtained from a Legendrian push off of $L_{i-1}$ by stabilizing $|a_i+2|$ times, for $i\geq 2$.
    \item The contact surgery coefficients are $+1$ on each $L_0^j$ and $-1$ on each $L_i$.
\end{itemize}
    
\end{theorem}

The choices we mentioned above correspond to the choices of stabilizations for each $L_i$, each of which can be either positive or negative. The case we are interested in is positive integer contact $+n$ surgery, and if we follow the algorithm carefully we can see that $+n$ contact surgery on a Legendrian knot $L$ is the same as doing contact surgery along the link $(L_0^1)\cup L_1 \cup...\cup L_{n-1}$ where $L_0^1$ is the Legendrian push off of $L$, $L_1$ is one stabilization of a Legendrian push off of $L_0^1$, and $L_i$ is a Legendrian push off of $L_{i-1}$ for $i\in\{2,3,...,n-1\}$. In particular there only one choice of stabilization required. The contact structure $\xi_n^-(L)$ we  consider corresponds to choosing the negative stabilization for $L_1$. Below is one reason why we want to consider the one with negative stabilizations. 

\begin{proposition}[\citep{LSct} Proposition 2.4] {\label{3.3}}
    Let $L$ be an oriented Legendrian knot and $L^-$ be negative stabilization of $L$, and let $n>0$ be a positive integer. Then $\xi_n^-(L)=\xi_{n+1}^-(L^-)$
\end{proposition}

Next we are going to see what happens when we take the $+n$ contact surgery on a Legendrian knot $L$ that is parallel to a binding of a compatible open book. Again we start with a Legendrian knot $L$ in a contact 3 manifold $(Y,\xi)$, and choose an abstract open book $(P,\phi)$ that is compatible with $(Y,\xi)$ containing $L$ on the page and has page framing equals to contact framing. We can get an open book that is compatible with $(Y_n(L), \xi_n^-(L))$ in the following way. First observe that by replacing $L$ and $(P,\phi)$ by a stabilization if necessary (as in \citep[Lemma 6.5]{BEHccm} or in \citep[Section 3]{MTn}), we can assume that $L$ is parallel to a boundary component of $P$ (by Proposition \ref{3.3} this does not lose generality). 

We then stabilize the open book again such that $L^-$ also lies on the page, and again by the observation in \citep[Lemma 6.5]{BEHccm} we may further assume $L^-$ is parallel to a binding component $B$ (though $L$ is no longer necessarily boundary parallel; see figure \ref{Contact surgery} for the description of the stabilization). 

We denote the stabilized open book $(P',k \circ \phi)$, where here and below we will use the same symbol for a simple closed curve on the page and the right-handed Dehn twist along that curve. By Theorem \ref{3.2} and the correspondence between surgery and Dehn twists, we conclude that $(P',\phi')=(P',(L^-)^{n-1} \circ (L)^{-1} \circ k \circ \phi)$ is an open book compatible with $(Y_n(L), \xi_n^-(L))$. 


\begin{figure}
    \centering    \includegraphics[width=\textwidth]{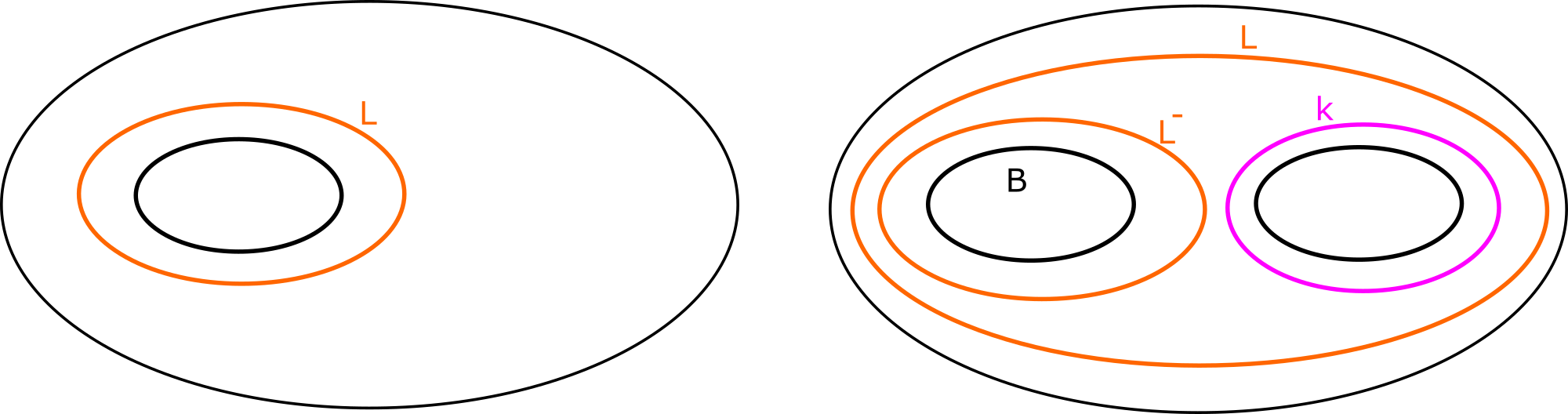}
    \caption{The left diagram describes the open book with Legendrian $L$ lie on it and parallel to some binding, and the right one is the stabilization of the left one (we do right hand twist along $k$), where $L^-$ now is parallel to some binding $B$. Then after we do $n-1$ right hand twists along $L^-$ and $1$ left hand twist along $L$ we obtain an open book $(P',\phi')$ for $(Y_n(L), \xi_n^-(L))$.
    }
    \label{Contact surgery}
\end{figure}

Now we are able to describe the important theorem that relates contact $+n$ surgery and capping off cobordism. 

\begin{theorem}[\citep{LSct} Proposition 4.1] {\label{3.4}}
In the above setting let $B_L$ be the binding of $(P',\phi')$ that corresponds to $B$ in $(P',k \circ \phi)$. Then 
\begin{itemize}
    \item Capping off $B_L$ gives us back $(P,\phi)$
    \item Let $X: Y_n(L)\to Y$ be the cobordism corresponding to capping off $B_L$, and let $W_{L,n}: Y\to Y_n(L)$ be the topological cobordism obtained by attaching a $4$-dimensional $2-handle$ along $L$ with framing $tb(L)+n$. Then, $X= -W_{L,n}$, i.e. $X$ is
obtained from $W_{L,n}$ by viewing it upside–down and reversing its orientation.

\end{itemize}
\end{theorem}

\section{Proof of Theorems \ref{th 1.1} and \ref{capoffthm}}
As we have seen in section 3, the contact $+n$ surgery is actually a special case of capping off, so instead of directly proving the $+n$ situation we will first prove the naturality result for the LOSS invariant under capping off cobordisms. In order to precisely state the proposition we need to first introduce a new definition.  

\begin{definition}
    Let $(M,\xi)$ be a contact 3-manifold, and $(P, \phi)$ an open book decomposition with at least 2 binding components supporting $(M,\xi)$. Consider $L$ a null-homologous oriented Legendrian knot, $B$ a binding component of $(P, \phi)$, and $\bf\{a_i\}$ a basis for $P$. We say a triple $(L,B,\bf\{a_i\})$ is adapted to $(P,\phi)$ if the following hold. 
\begin{enumerate}
    \item $L$ is sitting on the page $P$ and is parallel to some binding component $T$ other than $B$
    \item Up to reordering $\bf\{a_i\}$, $L$ intersects $a_1$ transversely at one point and does not intersect other $a_i$ for $i\neq 1$ 
    \item Up to reordering $\bf\{a_i\}$, $B$ intersects $a_2$ at exactly one point and does not intersect other $a_i$ for $i\neq 2$.
\end{enumerate}

We say $(L,B)$ is adapted to $(P,\phi)$ if condition $1$  holds. 
See Figure
\ref{adapted triple} for an example of an adapted triple. It's easy to see that if $(L,B)$ is adapted to $(P,\phi)$ then (maybe after further stabilization of the open book) we can always find basis $\bf\{a_i\}$ such that $(L,B,\bf\{a_i\})$ is adapted to $(P,\phi)$ 

\end{definition}

\begin{figure}[H]
    \centering
    \includegraphics[width=\textwidth]{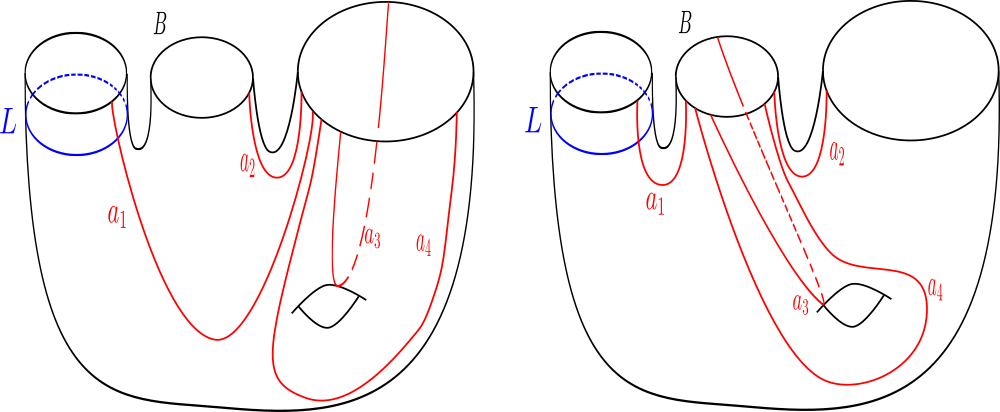}
    \caption{The left diagram is an adapted triple $(L,B,\bf\{a_i\})$ and the right one is not; one can transform the $\{a_i\}$ from one to the other by arcslides. 
    }
    \label{adapted triple}
\end{figure}

Similar to what we saw in section 2.2, given an adapted $(L,B,\bf\{a_i\})$ we can associate a doubly pointed Heegaard triple $(\Sigma,\alpha,\gamma,\beta,z,w)$ as follows. 

Let $\Sigma$ be the Heegaard surface that is the union of two pages $P_{+1} \cup -P_{-1}$, and consider the basis $\bf\{a_i\}$ as lying on $P_{+1}$. Let $c_i$  be a push off of $a_i$ for all $i$, and for all $i\neq 2$ let $b_i$  be a further push off of $c_i$. When $i = 2$ we let $b_2$ be a parallel push of the binding component $B$ on the page $P_{+1}$. We require that the push offs satisfy that each $a_i$, $b_i$, and $c_i$ intersect transversely at one point for all $i$ (as before, we arrange that the endpoints of the push off slide in the direction of the induced orientation of the boundary of $P_{+1}$). In particular, for all $i$, there is a ``small triangle'' formed by the arcs $a_i, c_i, b_i$, see Figure \ref{Double pointed Triple}.  

For the $\alpha$ and $\gamma$ curves in the Heegaard diagram we let $\alpha_i=a_i \cup \overline{a_i}$ and $\gamma_i=c_i\cup \overline{\phi(c_i)}$ for all $i$. For the $\beta$ curves let $\beta_i=b_i\cup \overline{\phi(b_i)}$ for $i\neq 2$, and $\beta_2=b_2$. Finally we place the base points $w,z$ such that they specify the Legendrian knot $L$ the same way as we define for LOSS invariant. 

Thus $(\Sigma, \alpha,\gamma, z,w)$ is a diagram for $(M, L)$, while $(\Sigma, \alpha, \beta, z, w)$ describes the induced knot $L'$ lying in the contact manifold $(M',\xi')$ obtained by capping off the binding component $B$ (this is clear after destabilizing the diagram using the single intersection between $\alpha_2$ and $\beta_2$). Furthermore, one can see (as in \citep{Bc}) that $(\Sigma,\alpha,\gamma,\beta,z,w)$ describes the capping off cobordism map from $M$ to $M'$ and send $L$ to $L'$. When we turn the cobordism upside down $(\Sigma,\bf{\gamma},\bf{\beta},\bf{\alpha})$ describes the cobordism map from $-M'$ to $-M$ (As in \citep{Bc},\citep{MTn}, $(\Sigma,\bf{\gamma},\bf{\beta},\alpha)$ is left-subordinate to this cobordism; see \citep{OShi} sections 4 and 5). After verifying admissibility conditions, this means that the doubly pointed Heegaard triple $(\Sigma,\gamma,\beta,\alpha,w,z)$ (Figure \ref{Double pointed Triple}) can be used to calculate the map \begin{equation}
    F_{B,\mathfrak{s}}: HFK^- (-M',L') \rightarrow HFK^-(-M,L)
\end{equation}
Let $ \bf{x}$ $=\{x_1,x_2,...x_g\}$, $ \bf{\Theta}$ $=\{\theta_1,\theta_2,...\theta_g\}$ and $ \bf{y}$ $=\{y_1,y_2,...y_g\}$ where $x_i=a_i\cap b_i$, $\theta_i=b_i\cap c_i$, and $y_i=a_i\cap c_i$ on $P_{+1}$. If we denote by $\Delta_i$ the small triangle connecting $x_i,\theta_i,y_i$ then the $Spin^c$ structure $\mathfrak{s}$ we wish to use in equation (4.1) is described by the Whitney triangle $\psi \in \pi_2(\bf{x,\Theta,y})$ (i.e. $\mathfrak{s}_z(\psi)=\mathfrak{s}$) where the domain $\mathbb{D}(\psi)$ is the sum of all $\Delta_i$ (\citep{OSht} Proposition 8.4). 
Then we have the following key proposition.

\begin{figure}
    \centering
    \includegraphics[width=\textwidth]{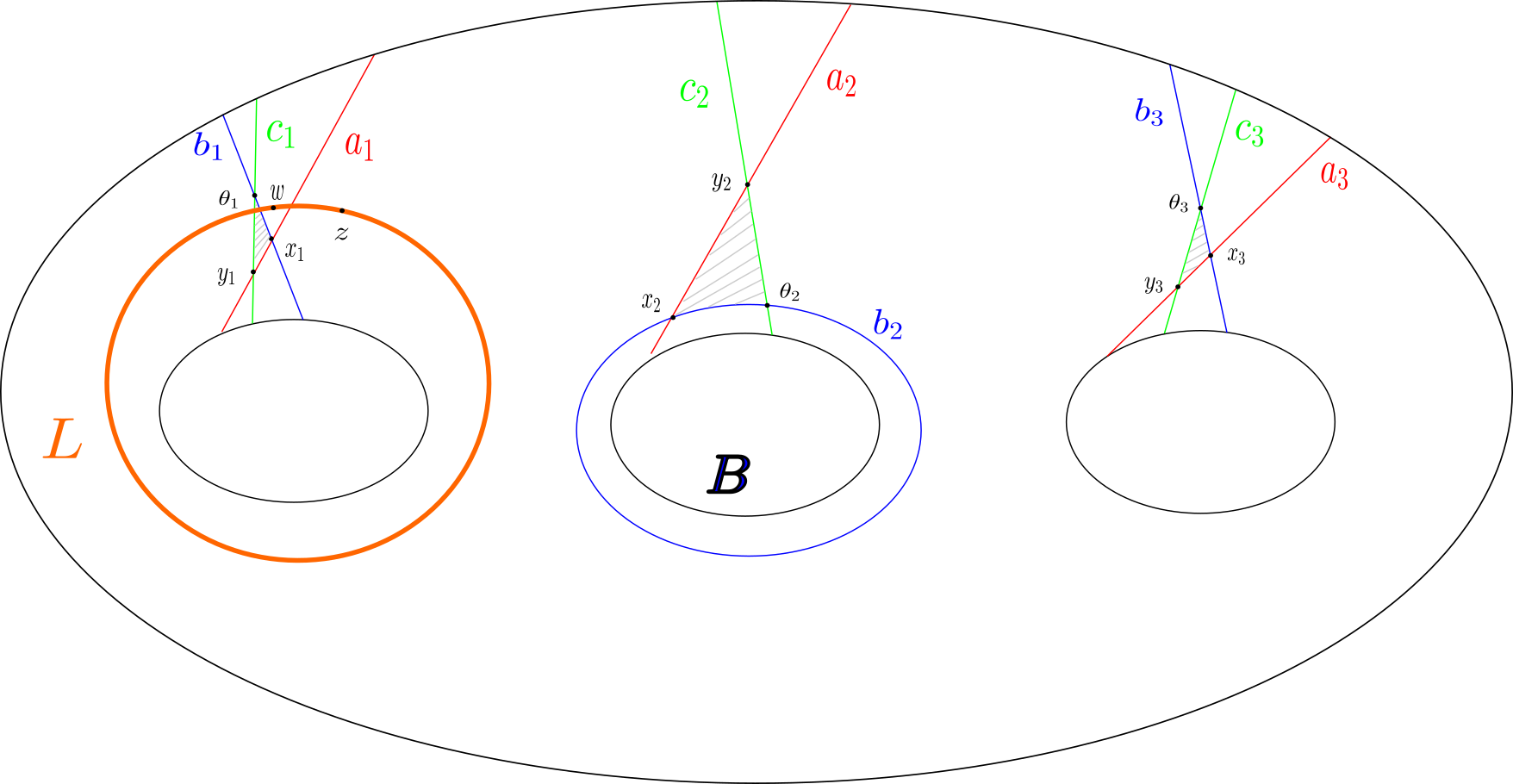}
    \caption{Since all of what we care are on the $P_{+1}$ page, we can just draw things on $P_{+1}$ to capture all the information instead of drawing the whole doubly pointed Heegaard triple. (The black circles are binding, red curves are $a_i$ (parts of the $\alpha_i$), blue curves are $b_i$ (parts of the $\beta_i$), green curves are $c_i$ (parts of the $\gamma_i$). The $Spin^c$ structure $\mathfrak{s}$ is represented by the small shaded triangle. There might be genus but it's not shown on the picture.) 
    }
    \label{Double pointed Triple}
\end{figure}

\begin{proposition} {\label{proposition 4.2}}
 Given adapted $(L,B,\bf\{a_i\})$ and map $F_{B,\mathfrak{s}}$ in the above setting. If we further assume $M'$ is a rational homology sphere and $L'$ is null-homologous in $M'$ then we have
 \begin{equation}
    F_{B,\mathfrak{s}}(\mathfrak{L}(M',\xi',L'))=\mathfrak{L}(M,\xi,L)
\end{equation}

\end{proposition}

We remark that by the choices we made in the push offs of $a_i$, 
the intersections $x_i$, $y_i$, and $\theta_i$ appear in  clockwise order around $\Delta_i$ for each $i$.

To prove the above Proposition we divide it into several lemmas. 

\begin{lemma} (cf. \citep[Lemma 2.2]{Bc})\label{Lemma 4.3}
    The doubly pointed diagram $(\Sigma,\gamma,\beta,\alpha,w,z)$ is weakly admissible, in the sense that any non-trivial triply-periodic domain has both positive and negative multiplicities.
\end{lemma}

Observe that if we ignore the $w$ base point then the local picture near $\Delta_i$ are all the same except for $i=2$. (See figure \ref{Local Triangles}, for local description of $\Delta_i$ and $\Delta_2$.) 

\begin{figure}
    \centering
    \includegraphics[width=\textwidth]{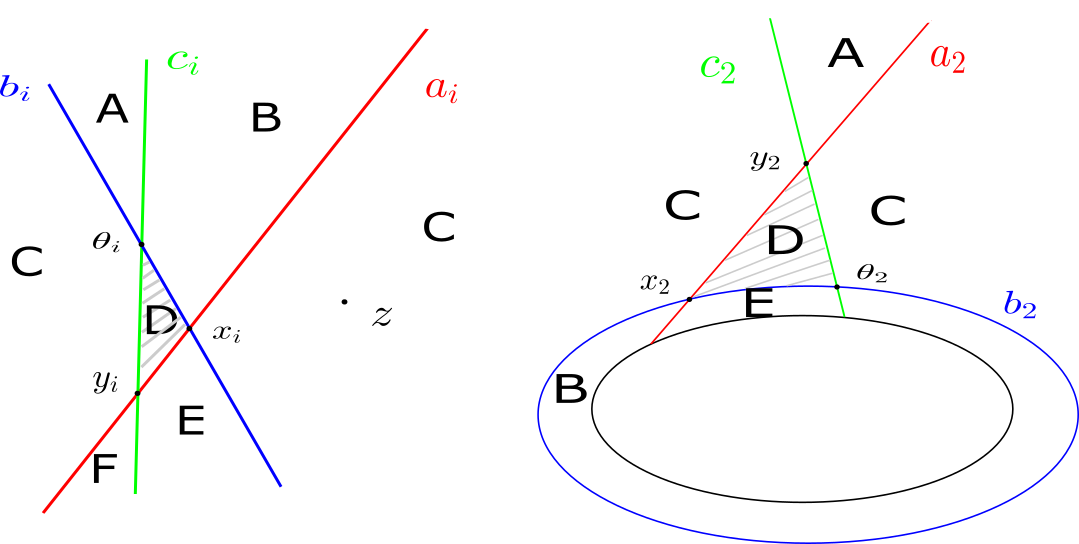}
    \caption{On the left is the local picture for $\Delta_i$ for $i\neq 2$, and on the right is the local picture for $\Delta_2$. A,B,C,D,E,F are the letters used to label the regions in local picture.
    }
    \label{Local Triangles}
\end{figure}

\begin{proof} 
 We first analyze the local picture for $\Delta_i$ where $i\neq 2$. Let $Q$ be a triply-periodic domain whose multiplicities in the regions $A$, $B$,
$C$, $D$, $E$ and $F$ are given by the integers $a$, $b$, $c$, $d$, $e$ and $f$, respectively. Since $\pa Q$ (rather, the portion of $\partial Q$ lying on the $\alpha$ circles) consists of full $\alpha$ arcs, we must have $$b-c=d-e=c-f.$$ Similarly $\pa Q$ also contains only full $\beta$ arcs so we have $$c-a=d-b=e-c.$$ Note that the region $C$ contains base point $z$ so $c=0$, which implies $Q$ has both positive and negative multiplicity unless  $$a=b=c=d=e=f=0.$$ Since for each $\Delta_i$ $i\neq 2$ the local pictures are the same, this tells us that if $Q$ has only positive or only negative multiplicities then $Q$ contains no $\alpha_i,\beta_i,\gamma_i$ as boundary for $i\neq 2$. So the only possibility for $Q$ to be nonzero is near $\Delta_2$. 

For the region around $\Delta_2$ we again label the regions $A$, $B$,
$C$, $D$, and $E$ as in Figure \ref{Local Triangles}, and the multiplicities are given by the integers $a$, $b$, $c$, $d$, and $e$ respectively. Again since $\pa Q$ contains only full $\alpha$ curves we must have $$a-c=c-d=b-e,$$ 
and since $C$ is the region containing base point $z$ we have $c=0$. Hence if $Q$ has only positive or only negative multiplicities then $a=d=0$, and $b=e$. If $b=e\neq 0$ we infer $\pa Q$  contains only the curve $\beta_2$, however since $\beta_2$ is not null-homologous in $\Sigma$ it can't bound a 2 chain by itself. So $b=e=0$ which shows the diagram is weakly admissible.

\end{proof}

\begin{lemma} (cf. \citep[Proposition 2.3]{Bc}) \label{Lemma 4.4}
     In the above setting, let $\bf{\bf{y'}}=(y'_1,y'_2,...,y'_g)\in \mathbb{T_\gamma}\cap \mathbb{T_\alpha}$ be an intersection point and $\psi' \in \pi_2(\bf{x},\bf{\Theta},\bf{y'})$ a Whitney triangle with only nonnegative local multiplicities. If $n_z(\psi')=0$ and $\mathfrak{s}_z(\psi')=\mathfrak{s}$, then $\bf{y'}=\bf{y}$, and $\psi'=\psi$.
\end{lemma}

\begin{proof}
We want to show that the domain of $\psi'$,  $\mathbb{D}(\psi')$ is the same as that of $\psi$, i.e.\ that $\mathbb{D}(\psi')=\mathbb{D}(\psi) = \Delta_1+\Delta_2+...+\Delta_g$.
As what we did above, we again look at what happens locally near $\Delta_i$ ($i\neq 2$) and $\Delta_2$. 

First we look around $\Delta_i$ ($i\neq 2$). Let $a$, $b$, $c$, $d$, $e$ and $f$ be the multiplicities of $\mathbb{D}(\psi')$ at $A$, $B$, $C$, $D$, $E$ and $F$. Since $\mathbb{D}(\psi')$ has corners at $x_i$ and $\theta_i$, we have \begin{equation}  
\tag{1}
    a+d=b+c+1, \ d+c=b+e+1 
\end{equation}

Since $c=0$, equations $(1)$ imply $a=-e$, and because the domain only contains non-negative multiplicities $a=e=0$. Therefore $(1)$ becomes 

\begin{equation}  \label{i}
\tag{2}
    d=b+1 
\end{equation}

Now if $y_i' \neq y_i$ for some $i\neq 2$ it implies $d+f=0$, so $d=f=0$, but when we put $d=0$ in $(2)$ we have $b=-1$ which is a contradiction. So $y_i=y'_i$, which implies $d+f=1$ (since $c=e=0$). If $d=0$ then combining with $(2)$ again we get $b=-1$, a contradiction. So we must have $d=1$ and $f=0$, which means that altogether $d=1$ and  $a=b=c=e=f=0$.  

Because the above argument works for all $i\neq 2$,  we conclude that $\mathbb{D}(\psi')$ is locally just $\Delta_i$ for $i\neq 2$, in other words $\mathbb{D}(\psi')=\Delta_1+\Delta_2'+\Delta_3+...+\Delta_g$, where $\Delta_2'$
is a region missing base point $z$ and whose oriented boundary consists
of arcs along $\beta_2$ from $\theta_2$ to $x_2$; along $\alpha_2$ from $x_2$ to $y'_2$; and along $\gamma_2$ from $y'_2$ to $\theta_2$. So we are left to show $y'_2=y_2$, and $\Delta'_2=\Delta_2$. 

Since $\mathfrak{s}_z(\psi')=\mathfrak{s}(\psi)=\mathfrak{s}$, by \citep[Proposition 8.5]{OSht}  we have $$\mathbb{D}(\psi')-\mathbb{D}(\psi)=\mathbb{D}(\phi_1)+\mathbb{D}(\phi_2)+\mathbb{D}(\phi_3)$$ where $\phi_1$, $\phi_2$, and $\phi_3$ are in $\pi_2(\bf{x},\bf{x})$, $\pi_2(\bf{\Theta},\bf{\Theta})$, and $\pi_2(\bf{y},\bf{y'})$ respectively. We want to show $\mathbb{D}(\psi')-\mathbb{D}(\psi)=0$.

First since $-M'$, which has Heegaard description $(\Sigma,\beta,\alpha)$, is a rational homology sphere, we have $\pi_2({\bf{x},\bf{x}})=0$ and therefore we can suppose $\mathbb{D}(\phi_1) = 0$ (strictly, $\pi_2({\bf{x},\bf{x}})$ always contains a copy of $\mathbb{Z}$ corresponding to multiples of the Heegaard surface, but these are not relevant  here because of positivity of multiplicities and the condition that $n_z(\psi) = n_z(\psi') = 0$). Moreover since $\mathbb{D}(\psi')-\mathbb{D}(\psi)=\Delta'_2 -\Delta_2$ and $\pi_2({\bf{x},\bf{x}})=0$, if $\mathbb{D}(\phi_2)\neq 0$ the ($\beta_2$ portion of the) boundary of $\mathbb{D}(\phi_2)$ can contain only multiples of $\beta_2$. However  $\beta_2$ is homologically independent from all linear combinations of $\beta_i$ for $i\neq 2$ and $\gamma_j$ for all $j$, and we conclude $\mathbb{D}(\phi_2)=0$. 

Therefore $$\mathbb{D}(\psi')-\mathbb{D}(\psi)=\Delta'_2 -\Delta_2=\mathbb{D}(\phi_3)$$
where $\mathbb{D}(\phi_3)$ is a domain not containing $z$, whose oriented boundary consists of arcs along $\alpha_2$ from $y_2$ to $y'_2$ and arcs along $\gamma_2$ from $y'_2$ to $y_2$. Since $\Delta_2$ has multiplicities  1 and 0 in the regions $D$ and $E$ respectively, the multiplicities of $\Delta'_2$ at $D$ and $E$ must satisfy \begin{equation}  
\tag{3}
    d-1=e.
\end{equation} At the same time, the boundary conditions at $x_2$ and $\theta_2$ tell us $$d+b-1=c+e.$$ Combining these two we get $b=c=0$. 

Now suppose $y_2\neq y_2'$. The boundary conditions then say $a+d=c+c=0$, so that $a=d=0$. Then when we return to $(3)$ we get $e=-1$, which is a contradiction---so $y_2=y_2'$. In this case the boundary constraint tells us $$a+d=1.$$ If we combine this with $(3)$ it follows $a=0=c=e=b$ and $d=1$. Hence $\Delta_2'=\Delta_2$, which implies $\mathbb{D}(\psi')-\mathbb{D}(\psi)=0$.

\end{proof}

Now we are ready to show Proposition \ref{proposition 4.2}.

\begin{proof}[Proof of Proposition \ref{proposition 4.2}]

 Lemma \ref{Lemma 4.3} says the map $F_{B,\mathfrak{s}}$ can be computed from the Heegaard diagram now under consideration. (Strictly, weak admissibility suffices to compute the homomorphism in the hat theory, while the minus theory requires strong admissibility for the $Spin^c$ structure under consideration. Weak and strong admissibility coincide if the $Spin^c$ structure is torsion on each boundary component. Alternatively, weak admissibility is also sufficient to define maps in the minus theory if we work over the power series ring ${\mathbb{F}}[[U]]$, so we work in that setting in the most general case.) Since the small triangle $\psi$ has a unique holomorphic representative, and Lemma \ref{Lemma 4.4} implies that the small triangle is the only one contributing to the map $F_{B,\mathfrak{s}}$, we have  $F_{B,\mathfrak{s}}(\bf{x})=\bf{y}$. We only left to show ${\bf{x}}=\mathfrak{L}(M',\xi',L')$, and ${\bf{y}}=\mathfrak{L}(M,\xi,L)$. The latter is clear by the definition of LOSS invariant . 

For $\mathfrak{L}(M',\xi',L')$, denote $(P_B,\phi_B)$ the corresponding open book after capping off binding $B$, and by abuse of notation consider $\{a_i\}$ (for $i\neq 2$) as a basis for $P_B$. By definition of LOSS again we can  see that $\mathfrak{L}(M',\xi',L')$ is represented by ${\bf{x'}}=(x_1,x_3,...,x_g)$. The diagram $(\Sigma, {\bf{\beta},\bf{\alpha}}, z, w)$ then differs from the one obtained from $P_B$ by a stabilization of the Heegaard diagram. Then by \citep[Section 10]{OSht}, we see under the isomorphism induced by the stabilization we map $\bf{x'}$ to $\bf{x}$. So ${\bf{x}}=\mathfrak{L}(M',\xi',L')$.
\end{proof}

\begin{remark}
    Since the Legendrian knot $L$ is a Legendrian approximation of $T$ ($L$ is parallel to $T$ in the open book), the Theorem \ref{capoffthm} follows. 
\end{remark} 

Now we are ready to prove Theorem \ref{th 1.1}. 

\begin{proof}[Proof of Theorem 1.1]
We first choose an arbitrary open book $(P,\phi)$ supporting $(Y,\xi)$ and having $L$ and $S$ on the page $P$, where as we saw in section \ref{Contact surgery and Capping off cobordism} we may further assume $S$ is parallel to some binding of $P$. Again by \citep[Lemma 6.5]{BEHccm} we can stabilize the open book such that $S^-$, the negative stabilization of $S$, is parallel to a binding component $B$, and $L^-$ is also parallel to some other binding component $T$. We again call the stabilized open book $(P,\phi)$. 

Now by Theorem \ref{3.4} the smooth cobordism from $-Y$ to $-Y_n(S)$ induced by contact $n$ surgery (smooth $tb(S)+n$) is the same as a capping off cobordism from $Y_n(S)$ to $Y$ viewed upside down, where we cap off a binding component $B_S\subset Y_n(S)$ as we saw in section \ref{Contact surgery and Capping off cobordism}.  (Note here that $S$ is playing the role of $L$ in section \ref{Contact surgery and Capping off cobordism} and the above parts of this section.)

 In other words we have an open book $(P_S,\phi_S)$ for $(Y_n(S),\xi_{n}^-(S))$ such that capping off $B_S$ gives us back $(P,\phi)$, and such that the knot induced by $({L^{-}})_S$ is $L^-$. Then (possibly after further stabilization of $(P_S,\phi_S)$) we choose a basis $\bf{\{a_i\}}$ such that $(({L^{-}})_S,B_S,\bf{\{a_i\}})$ is adapted to $(P_S,\phi_S)$, and thus by Proposition \ref{proposition 4.2} $$F_{B_S,\mathfrak{s}}(\mathfrak{L}(Y,\xi,L^-))=\mathfrak{L}(Y_n(S),\xi^-_{n}(S),(L^{-})_S).$$ By equivalence of contact surgery and capping off we have  $$F_{S,\mathfrak{s}}(\mathfrak{L}(Y,\xi,L^-))=\mathfrak{L}(Y_n(S),\xi^-_{n}(S),(L^{-})_S)$$
 
 Finally because $(L^-)_S =(L_S)^-$, and the LOSS invariant is invariant under negative stabilization (Theorem \ref{2.1}) we conclude that $$F_{S,\mathfrak{s}}(\mathfrak{L}(Y,\xi,L))=\mathfrak{L}(Y_n(S),\xi^-_{n}(S),L_S)$$   
\end{proof}

\section{Application}

One interesting application of the main theorem is to give many more examples of Legendrian and transversely non-simple knots which are distinguished by their LOSS invariants, as stated in Theorem \ref{th 1.5} 

\subsection{Theorem \ref{th 1.5}}

To prove Theorem \ref{th 1.5} we need the following lemma. 

\begin{lemma}
{\label{lemma 5.1}}
    Let $L$ be an oriented Legendrian knot in $(Y,\xi)$, and for $i=1,2,...n$ let $e_i$ be arcs of $L$ such that they are horizontal parallel with the same orientation inside some Darboux ball $B$. Moreover let $S$ be an oriented max tb unknot in $B$ that links each $e_i$ positively once (so the linking number between $S$ and $L$ is $+n$). Then after doing $+2$ contact surgery on $S$, with the choice of stabilization being negative, the resulting contact manifold is contactomorphic to $(Y,\xi)$, but the resulting $e_i$'s are parallel Legendrian push offs of a negative stabilization of $e_1$ (thus, smoothly, the new strands have a full negative twist).
\end{lemma}
\begin{figure}[H]
    \centering
    \includegraphics[width=\textwidth]{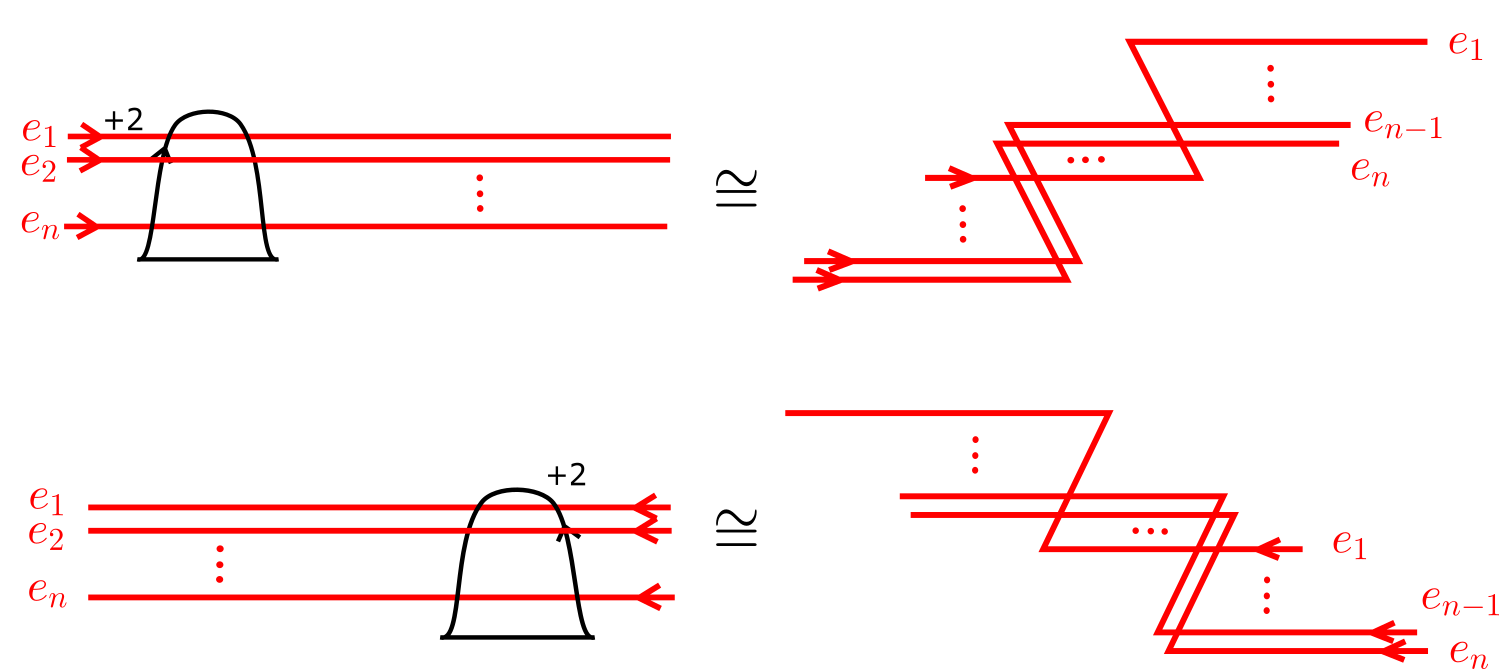}
    \caption{there are $n$ red arcs $e_1$ to $e_n$
    }
    \label{2 different orientations}
\end{figure}
\begin{proof}
    There are two possibilities of how those $e_i$ are oriented, either from left to right or from right to left. So to prove the Lemma it's the same to show the equivalence of the pair of contact surgery diagrams in Figure \ref{2 different orientations}. 

    Since the proof is symmetric with one and the other we will only show top case of the figure (strand orientation from left to right). Note that since the $e_i$ are Legendrian push offs of each other, it's enough to consider the case when $n=1$. In Figure \ref{+2 contact surgery} we exhibit a sequence of Legendrian isotopies, contact surgery and contact handle moves  to show the equivalence of the two diagrams when $n=1$. 

    \begin{figure}
    \centering
    \includegraphics[width=\textwidth]{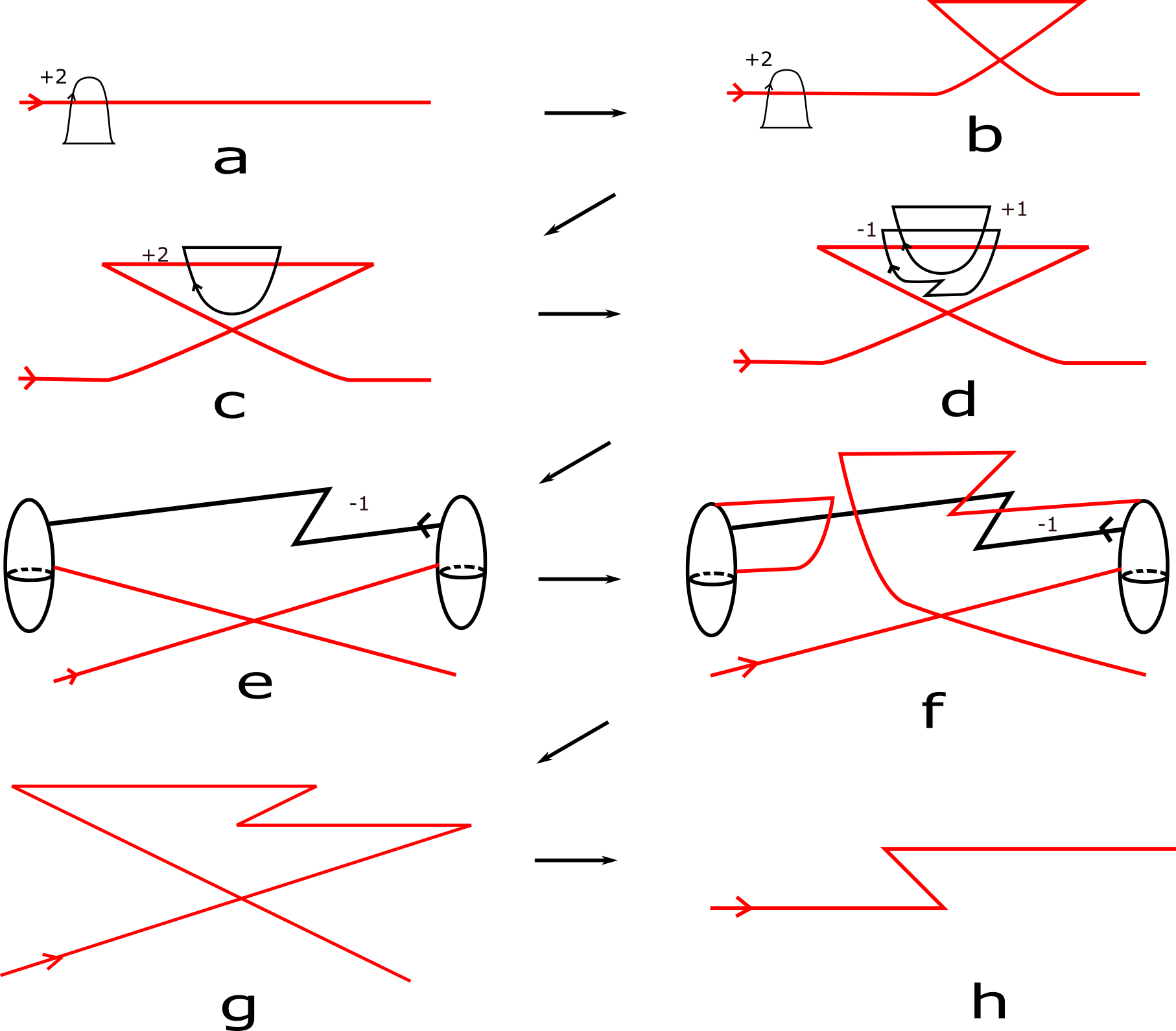}
    \caption{ From \textbf{a} to \textbf{b} we use Legendrian Reidemeister 1 moves; from \textbf{b} to \textbf{c} we isotopy Legendrian meridian from bottom to top using \citep[Figure 13-15]{DGh}; from \textbf{c} to \textbf{d} we use the DGS algorithm \citep{DGS} to change $+2$ contact surgery to $+1$ and $-1$ contact surgeries, and we use negative stabilization as the assumption; from \textbf{d} to \textbf{e} we use \citep[Theorem 4]{DGh}  to identify surgery diagram with handle diagram; from \textbf{e} to \textbf{f} we handle slide the red curve over the $-1$ framed handle using \citep[Proposition 1]{DGh}; from \textbf{f} to \textbf{g} we cancel out the $-1$ framed 2 handle with the 1 handle; and last we perform a Legendrian Reidemeister move to get rid of the extra crossing and attain \textbf{h}.
    }
    \label{+2 contact surgery}
\end{figure}
\end{proof}

To prove Theorem \ref{th 1.5} notice that given a compatible triple $(L,\sigma_n, B)$ the new Legendrian knot $L_\sigma$ we form only differs from $L$ by a positive twist. We intend to use the above Lemma on $L_\sigma$, so we can cancel out the positive twist with a negative twist and give back $L$. With this idea in mind we are ready to start the proof.

\begin{proof}[Proof of Theorem \ref{th 1.5}]

Since there exist a smoothly isotopy from $L$ to $L'$ that takes the $B$ to $B'$, and inside the balls we are doing the same operation to the arcs, we infer the resulting knots $L_\sigma$ and $L'_{\sigma'}$ are smoothly isotopic, proving the first part.  

Then let's consider the Darboux ball $B$ and the new Legendrian knot $L_\sigma$. Inside the ball the arcs $e_i$ have the same orientation and can be considered to be initially horizontal and parallel (near the left side of the ball),  then they start doing a positive twist as we move from left to right. Now as in Lemma 5.1 let $S$ be an oriented max tb unknot that links both $e_i$ positively one time, and perform $+2$ contact surgery on $S$. We can think of this as happening near the horizontal parallel part of the $e_i$, so it looks like the top left diagram of figure \ref{untwist}. By Lemma \ref{lemma 5.1} this is equivalent to the top right of figure \ref{untwist}, then after sequence of Legendrian Reidemeister moves it's not hard to see we obtain the bottom right. 

\begin{figure}[H]
    \centering
    \includegraphics [width=\textwidth]{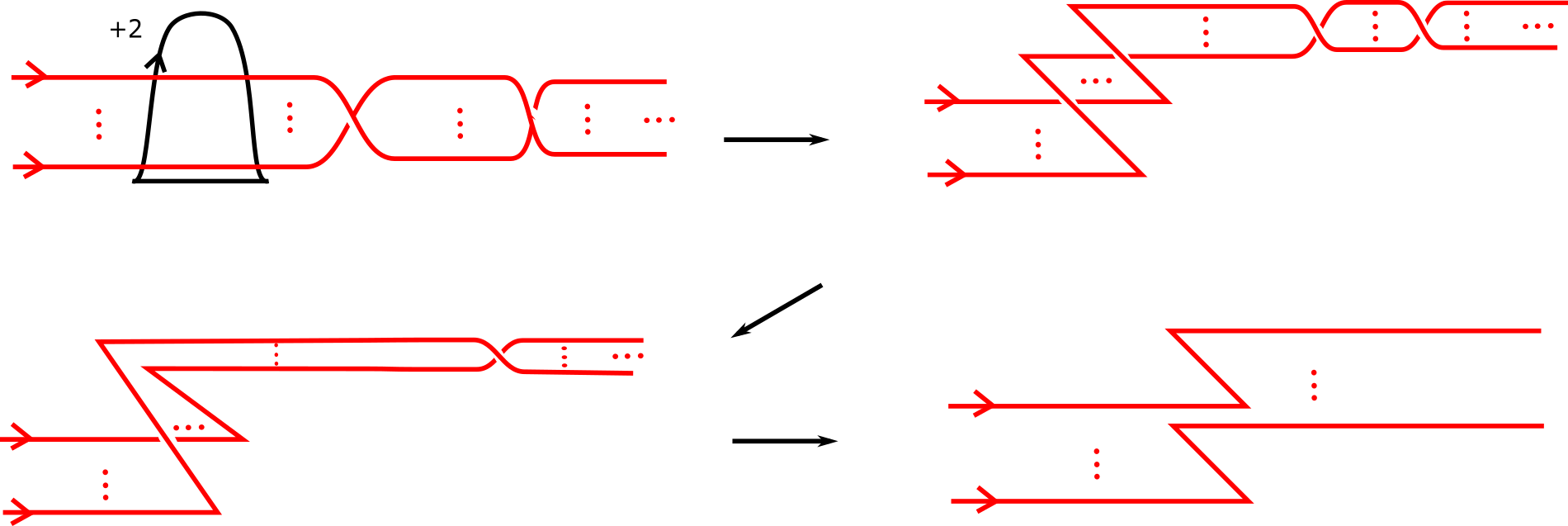}
    \caption{We apply the Lemma \ref{lemma 5.1} on the arcs in $L_\sigma$, then we undo the twist which gives back $L$ with $n$ additional negative stabilization. 
    }
    \label{untwist}
\end{figure}

From the picture we can easily see that  doing this $+2$ contact surgery on $S$ transforms $L_\sigma$ to the $n$-fold negative stabilization $L^{-n}$ of $L$. Now we want to apply Theorem \ref{th 1.1}. Since $L$ is null-homologous, by construction $L_\sigma$ is also null-homologous (we can add bands to the Seifert surface of $L$), and since we do not change the ambient contact 3 manifold $(Y,\xi)$ by doing $+2$ contact surgery, the map in Theorem \ref{th 1.1} is of the form  

\begin{equation}
    F_{S,\mathfrak{s}}: HFK^- (-Y,L_\sigma) \rightarrow HFK^-(-Y,L).
\end{equation}
Since $Y$ is a rational homology sphere, and LOSS invariant is unchanged under negative stabilization, by Theorem \ref{th 1.1} there exist some $Spin^c$ structure $\mathfrak{s}$ such that 

\begin{equation}\label{natural1}
    F_{S,\mathfrak{s}}(\mathfrak{L}(L_\sigma))=\mathfrak{L}(L^{-n})=\mathfrak{L}(L)
\end{equation}

We do the same thing, $+2$ contact surgery on $S'$ for arcs $e'_i$ of $L_{\sigma'}'$ inside $B_{\sigma'}$. So we also get a map 

\begin{equation}
    F_{S',\mathfrak{s'}}: HFK^- (-Y,L'_{\sigma'}) \rightarrow HFK^-(-Y,L')
\end{equation}
such that \begin{equation}\label{natural2}
    F_{S',\mathfrak{s'}}(\mathfrak{L}(L'_{\sigma'}))=\mathfrak{L}( L'^{-n})=\mathfrak{L}(L')
\end{equation}

By the construction of $L_\sigma$ and $L'_{\sigma'}$ and the assumption about the isotopy from $L$ to $L'$ it's easy to see that there also exist a smooth isotopy from $L_\sigma$ to $L_{\sigma'}'$ that sends $S$ to $S'$. Moreover since $S$ and $S'$ are null-homologous (they are inside the ball) and $\xi$ has torsion first Chern class ($Y$ is a rational homology sphere), by Proposition \ref{prop 1.4} 
$$\langle c_1(\mathfrak{s}),[\Tilde{Z}] \rangle = rot(S)+1$$ and 
$$ \langle c_1(\mathfrak{s'}),[\Tilde{Z}] \rangle = rot(S')+1$$

Note that both $\mathfrak{s}$ and $\mathfrak{s}'$ restrict to the $Spin^c$ structure corresponding to $\xi$ on both boundaries of the cobordism. This condition together with the value of $\langle c_1({\mathfrak{s}}), [\Tilde{Z}]\rangle$ determines a $Spin^c$ structure uniquely on the surgery cobordism. Thus, because $rot(S)=rot(S')$, we infer $\mathfrak{s}$ is equal to $\mathfrak{s'}$. 

Let us write $K$ and $K_\sigma$ for the smooth knot types underlying $L$, $L'$ and $L_\sigma$, $L'_{\sigma'}$, respectively. By the results of \citep{JTZnmh}, we can consider the LOSS invariants of $L$ and $L'$ to lie in the same group $HFK^-(-Y, K)$, and similarly those of $L_\sigma$ and $L'_{\sigma'}$ lie in $HFK^-(-Y, K_\sigma)$. More precisely, this means that there are canonical isomorphisms
\begin{align*}
HFK^-(-Y, L_\sigma)&\to HFK^-(-Y, K_\sigma), \qquad HFK^-(-Y, L'_{\sigma'})\to HFK^-(-Y, K_\sigma)\\
HFK^-(-Y, L)&\to HFK^-(-Y, K), \qquad HFK^-(-Y, L')\to HFK^-(-Y, K).
\end{align*}
With these identifications in mind, we will drop the distinction between the circles $S$ and $S'$, as they are ambiently smoothly isotopic. 




 Now for any $d\in MCG(Y,K_\sigma)$, let $d(S)$ denote the induced knot and $d_S\in MCG(Y, K)$  the induced diffeomorphism after the surgery on $S$; moreover let $d^*$,$d_S^*$ be the induced maps on knot Floer homology and $d(\mathfrak{s})$ be the induced $Spin^c$ structure. Then by Theorem 8.9 and Corollary 11.17 in \citep{Jcsf}, and Theorem 1.8 in \citep{JTZnmh} we have the following commutative diagram. 

 \begin{equation}
      \begin{tikzcd}
	{HFK^- (-Y,K_\sigma)} &&&& {HFK^-(-Y,K)} \\
	{HFK^- (-Y,K_{\sigma})} &&&& {HFK^-(-Y,K)}
	\arrow["{F_{S,\mathfrak{s}}}", from=1-1, to=1-5]
	\arrow["d^*", from=1-1, to=2-1]
	\arrow["d_S^*", from=1-5, to=2-5]
	\arrow["{F_{d(S),d(\mathfrak{s})}}"', from=2-1, to=2-5]
\end{tikzcd} \label{CD2}
 \end{equation}
 which implies
 
 $$d_S^*(F_{S,\mathfrak{s}}(\mathfrak{L}(L_\sigma)))= F_{d(S),d(\mathfrak{s})} (d^*(\mathfrak{L}(L_\sigma))).$$
 Since $F_{S,\mathfrak{s}}(\mathfrak{L}(L_\sigma)) = \mathfrak{L}(L)$, we have

$$d_S^*(\mathfrak{L}(L))= F_{d(S),d(\mathfrak{s})} (d^*(\mathfrak{L}(L_\sigma))).$$

Our assumption is that $[\mathfrak{L}(L)]\neq [\mathfrak{L}(L')]$ (strictly, that these $MCG$ orbits are different under the canonical isomorphisms above). Now suppose $[\mathfrak{L}(L_\sigma)]=[\mathfrak{L}(L'_{\sigma'})]$, so that there exists $d\in MCG(Y,L_\sigma)$ such that $d^* (\mathfrak{L}(L_\sigma)) =\mathfrak{L}(L'_{\sigma'})$. Combined with the above, we infer
\[
d_S^*(\mathfrak{L}(L))= F_{d(S),d(\mathfrak{s})}(\mathfrak{L}(L'_{\sigma'})).
\]
Now we claim that this class is the same as $F_{S, \mathfrak{s}}(\mathfrak{L}(L'_{\sigma'}))$. To see this, note first that implicit in the condition that $d^*(\mathfrak{L}_\sigma) = \mathfrak{L}'_{\sigma'}$ is the requirement that $d_*(\xi) = \xi$. Moreover, since we are free to modify $d$ by an isotopy (fixing $K_\sigma$), we can suppose that $d$ is the identity on the ball containing $S$. 
Since $d$ preserves the contact structures it must fix the induced $Spin^c$ structures on the boundary $-Y\sqcup -Y$. As the Chern number evaluation on the cobordism is also preserved, we infer $(d(S), d(\mathfrak{s})) = (S,\mathfrak{s})$. By the naturality theorem, it then follows that $F_{d(S), d(\mathfrak{s})}(\mathfrak{L}'_{\sigma'}) = F_{S', \mathfrak{s}'}(\mathfrak{L}'_{\sigma'}) =\mathfrak{L}(L')$. From the equation above, we obtain  $d_S^*(\mathfrak{L}(L)) = \mathfrak{L}(L')$, contrary to assumption.


 

Exactly same arguments work for  $\mathfrak{\widehat{L}}$.
\end{proof}

\subsection{Non-simplicity of Legendrian and transverse knot} \label{section 5.2}

Let's see an example of non-simple knot using Theorem $\ref{th 1.5}$ 

\begin{proof} [Proof of Corollary \ref{cor 1.6}]

It's easy to see the two Legendrian knots in figure \ref{m(9_7)} are smoothly isotopic to $m(9_7)$ and have same $tb$ and $rot$. We claim that they have different $[\mathfrak{\widehat{L}}]$, which will imply the two are not Legendrian isotopic, and also that their transverse push offs are not transverse isotopic. The Legendrians in Figure \ref{m(9_7)} were obtained by an application of (the construction leading to) Theorem \ref{th 1.5} to the two knots in figure \ref{m(7_2)}. According to \citep[Theorem 1.3] {OSct} the two knots in Figure \ref{m(7_2)} have different $[\mathfrak{\widehat{L}}]$; moreover we can smoothly isotop the left side of Figure \ref{m(7_2)} to the right while fixing everything in the green circle. This verifies the assumptions of Theorem \ref{th 1.5}, thus after adding a full twist to the arcs in green circle, two new Legendrian knots in figure \ref{m(9_7)} have different $[\mathfrak{\widehat{L}}]$ invariant. 

\begin{figure}[H]
    \centering
    \includegraphics{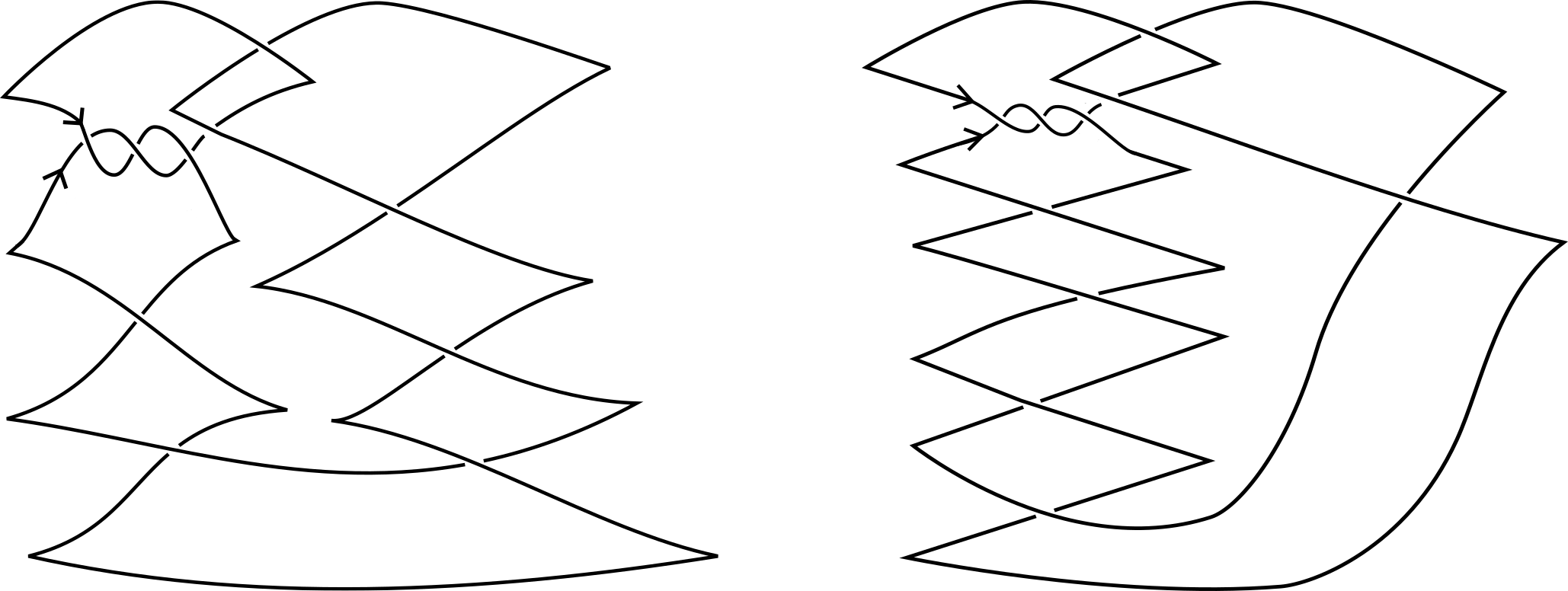}
    \caption{Both of these are smoothly isotopic to $m(9_7)$}
    \label{m(9_7)}
\end{figure}

\begin{figure} [H]
    \centering
    \includegraphics{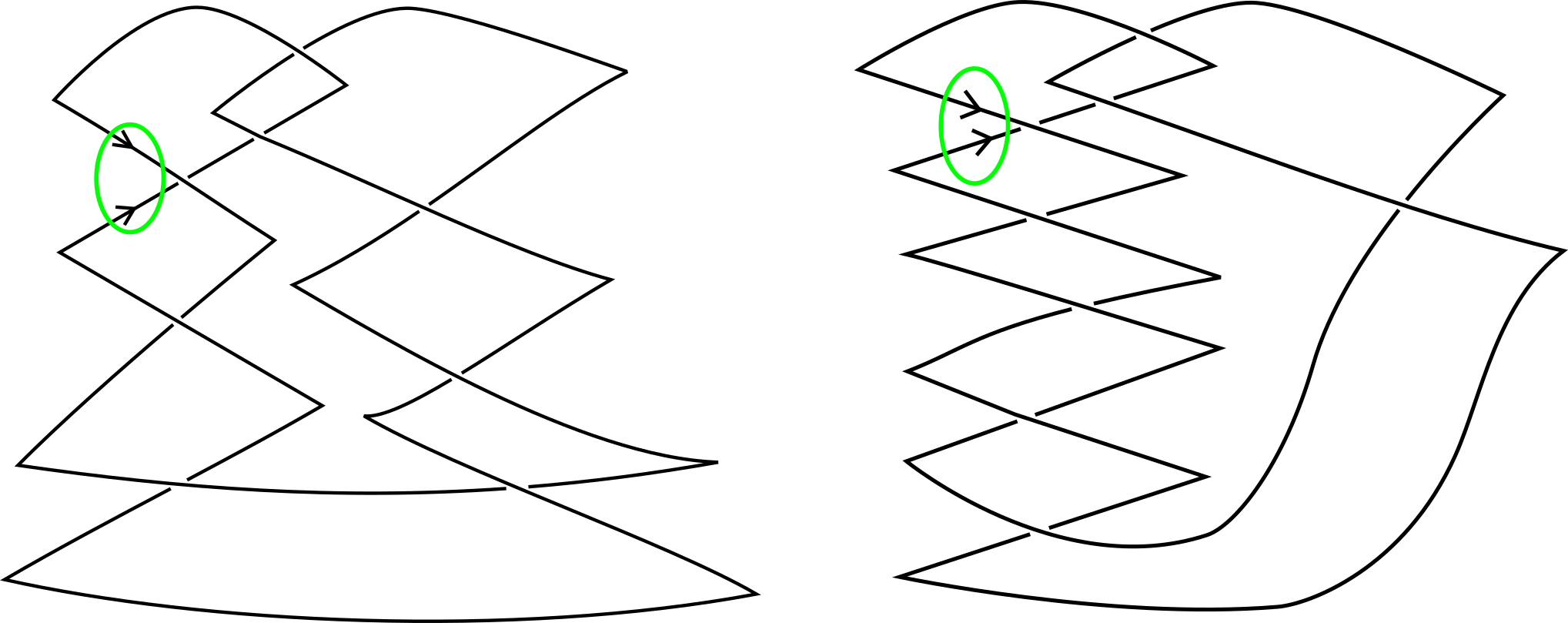}
    \caption{Two different Legendrian representatives of the Eliashberg–Chekanov twist knot $E_5=m(7_2)$. The green circle indicate where we apply Theorem \ref{th 1.5}}
    \label{m(7_2)}
\end{figure} \end{proof}

Notice that the knot $m(9_7)$ is a rational knot and in Conway's notation \citep{Cek} it is the [-3,-5,2] knot. Using similar ideas as the above, we can get infinite families of knot that are Legendrian and transversely non-simple. 

\begin{theorem} 
{\label{th 5.2}}
    Let $m,n$ be positive integers with $n>3$ and odd. In Conway's notation the knot  $[-2m-1,-n,2]$ (Figure \ref{two bridge knot}) has at least $\lceil \frac{n}{4} \rceil$ Legendrian (transverse) representatives that have  $tb=2m+1$ and $rot=0$ (self-linking number $2m+1$) that are pairwise not Legendrian (transverse) isotopic.
\end{theorem}

\begin{figure}[H]
    \centering
    \includegraphics{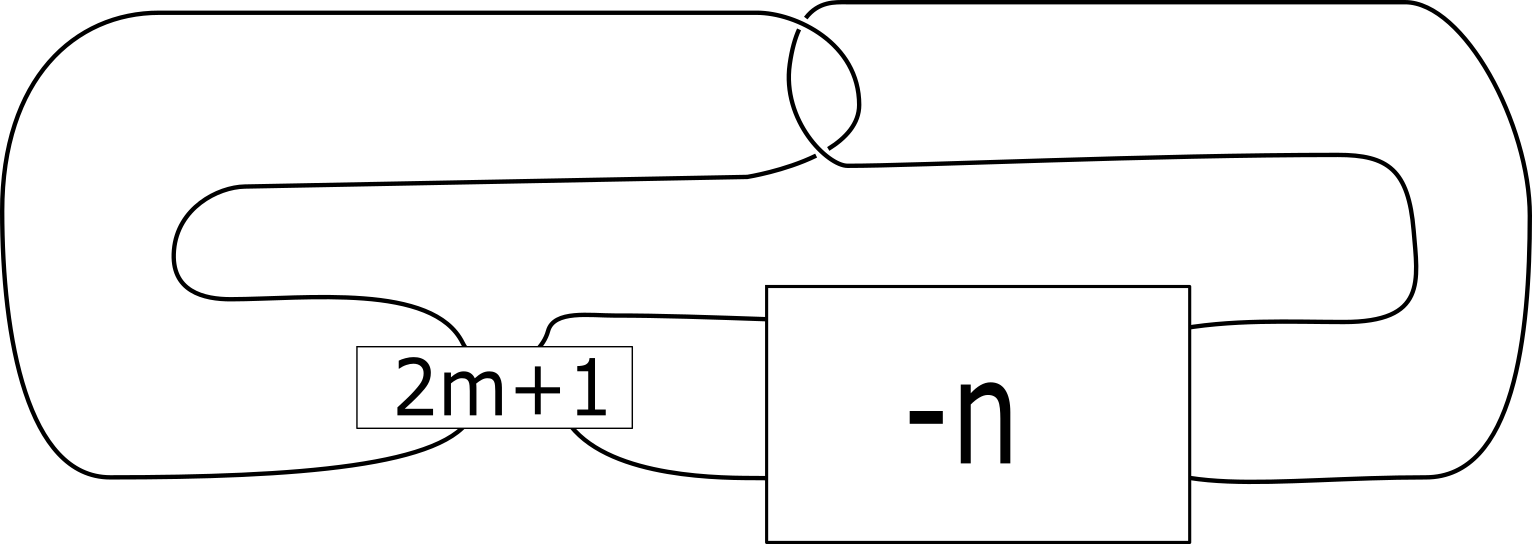}
    \caption{+1 means one right handed half twist and -1 means one left handed half twist. For  $m=1$ $n=5$ the result is $m(9_7)$; if $m=2$  $n=5$ we have $m(11a_{242})$; if  $m=1$ $n=7$ we have $m(11a_{246})$
    }
    \label{two bridge knot}
\end{figure}

\begin{proof}

Again in \citep[Theorem 1.3]{OSct} Ozsv\'ath and Stipsicz prove that the Eliashberg–Chekanov twist knot $E_n$ shown in figure \ref{twist knot} has $\lceil \frac{n}{4} \rceil$ many Legendrian representatives $E(k,l)$ (see figure \ref{Legendrian E_n}), moreover $E(k,l)$ and $E(k'.l')$ have different $[\mathfrak{\widehat{L}}]$ provided that $k,l,k',l'$ are odd, $k+l-1=k'+l'-1=n$, $k\geq l$, $k'\geq l'$, and $k\neq k'$.

So similar to the proof of Corollary \ref{cor 1.6}, to construct Legendrian $[-2m-1,-n,2]$ we just apply Theorem \ref{th 1.5}  $m$ times to all pairs of $E(k,l)$, in the Darboux ball represented by green circle in figure \ref{Legendrian E_n}. Notice that each time we apply the Theorem \ref{th 1.5} we add one full right handed twist to the green circle in figure \ref{twist knot}. So we still have $\lceil \frac{n}{4} \rceil$ many representative of the new knot $[-2m-1,-n,2]$ that have pairwise distinct $[\mathfrak{\widehat{L}}]$. Moreover, $E(k,l)$ all have $rot=0$ and $tb=1$, and adding a positive twist to $E(k,l)$ does not change the rotation number and adds two to the Thurston–Bennequin number. So all those new representative have the same $tb=2m+1$ and $rot=0$. Since these representatives are distinguished by $[\mathfrak{\widehat{L}}]$, their transverse push offs are also not transversely isotopic and have self-linking number $2m+1$.
\end{proof}

\begin{figure}[H]
    \centering
    \includegraphics{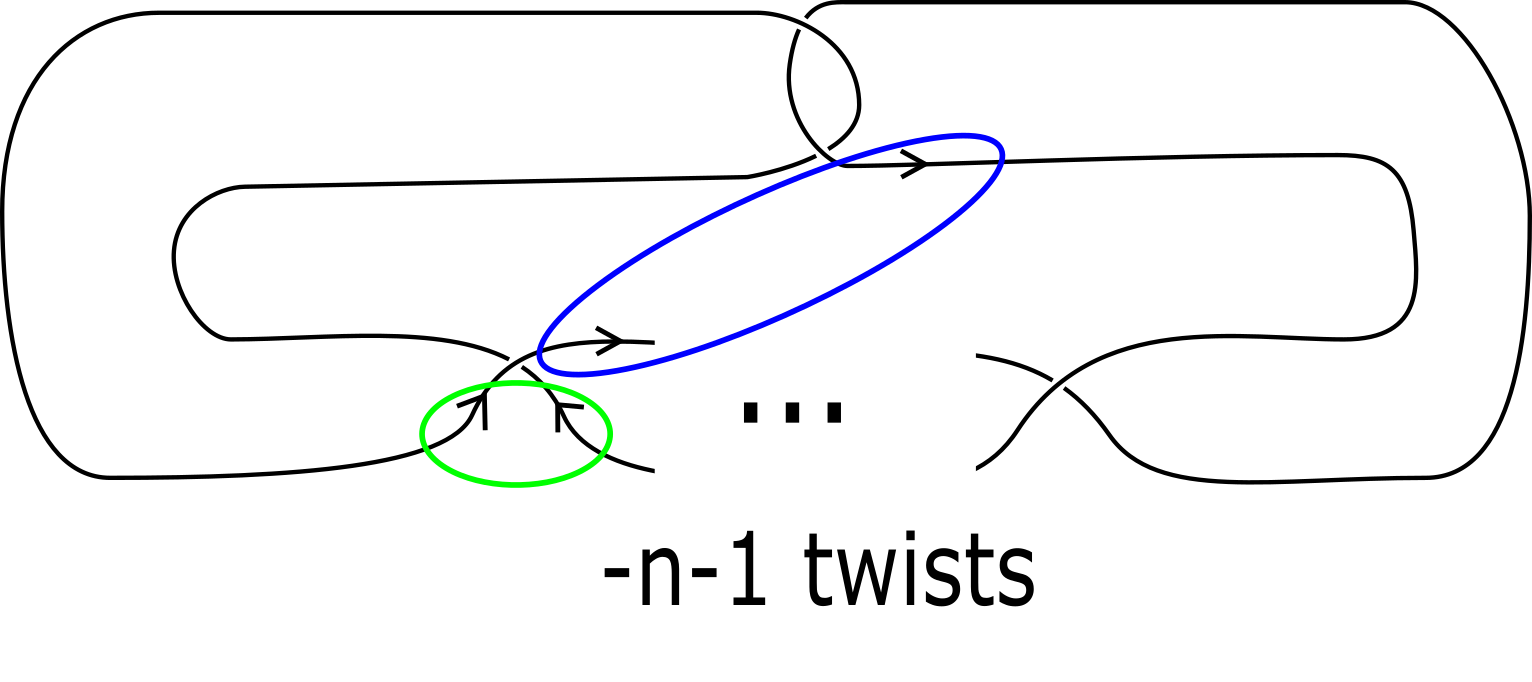}
    \caption{The Eliashberg–Chekanov twist knot $E_n$. The green (blue) circle smoothly corresponds to the green (blue) circle on the different Legendrian realizations of $E(k,l)$ of $E(n)$ in Figure \ref{Legendrian E_n}.}
    \label{twist knot}
\end{figure}

\begin{figure}[H]
    \centering
    \includegraphics[scale=0.88]{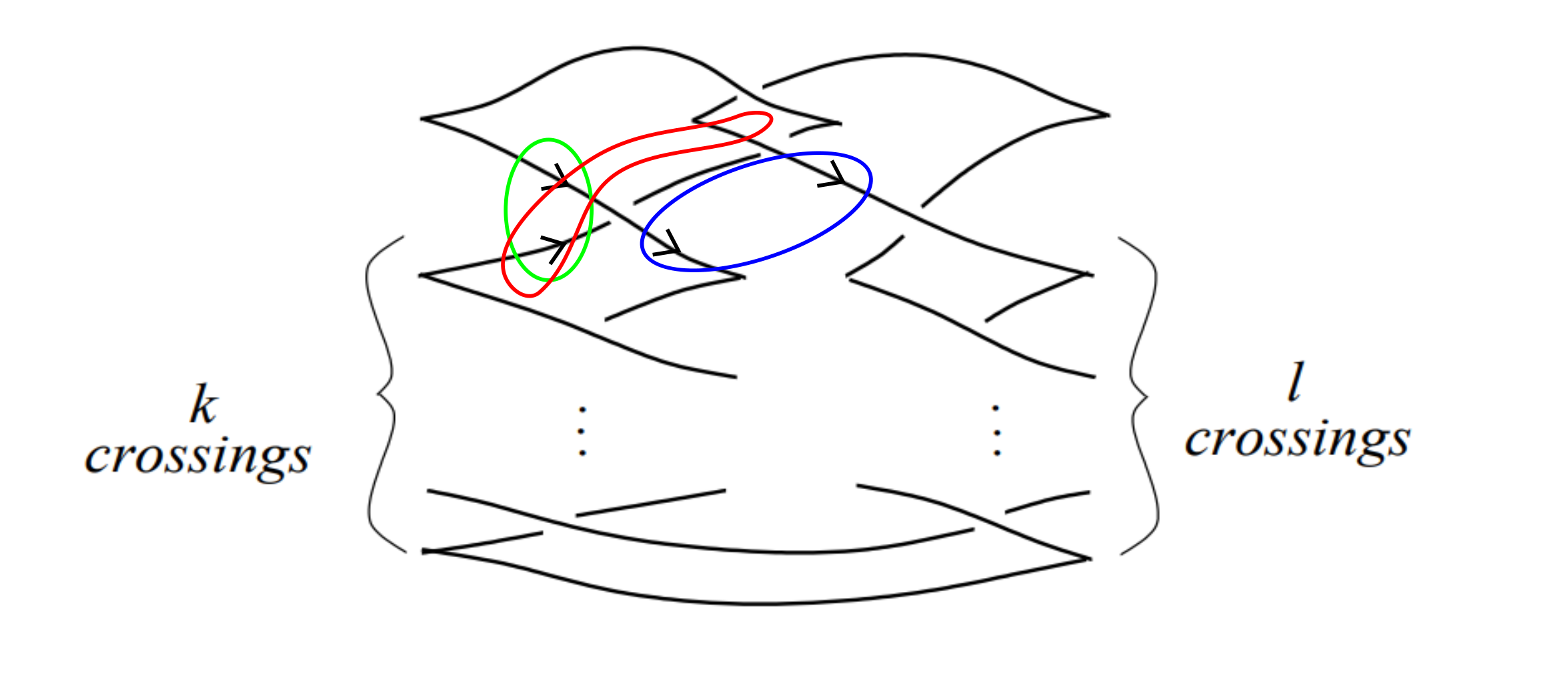}
    \caption{ The Legendrian knots $E(k,l)$, $k,l\geq 1$ odd. These knots are smoothly isotopic
to $E_n$, with $k+l-1 = n$. The green, blue or red circle indicates possible place where we can apply the Theorem \ref{th 1.5}. (Diagram from Figure 8 in \citep{OSct} )
    }
    \label{Legendrian E_n}
\end{figure}

Instead of applying Theorem \ref{th 1.5} to the green circle on $E(k,l)$, we can also apply it to the blue circle. The exact same proof applies, and it will give us a family of double twist knot that are non-simple.

\begin{theorem}

{\label{th 5.3}}
    Let $m,n$ be positive integers with $n > 3$ and odd. The double twist knot $K(2m+2,-n)$ (Figure \ref{double twist knot}, the knot $[-n,-2m-2]$ in Conway notation) has at least $\lceil \frac{n}{4} \rceil$ Legendrian (transverse) representatives that have  $tb=2m+1$ and $rot=0$ (self-linking number $2m+1$) and are pairwise not Legendrian (transverse) 
    isotopic.
\end{theorem} 

\begin{figure}[H]
    \centering
    \includegraphics{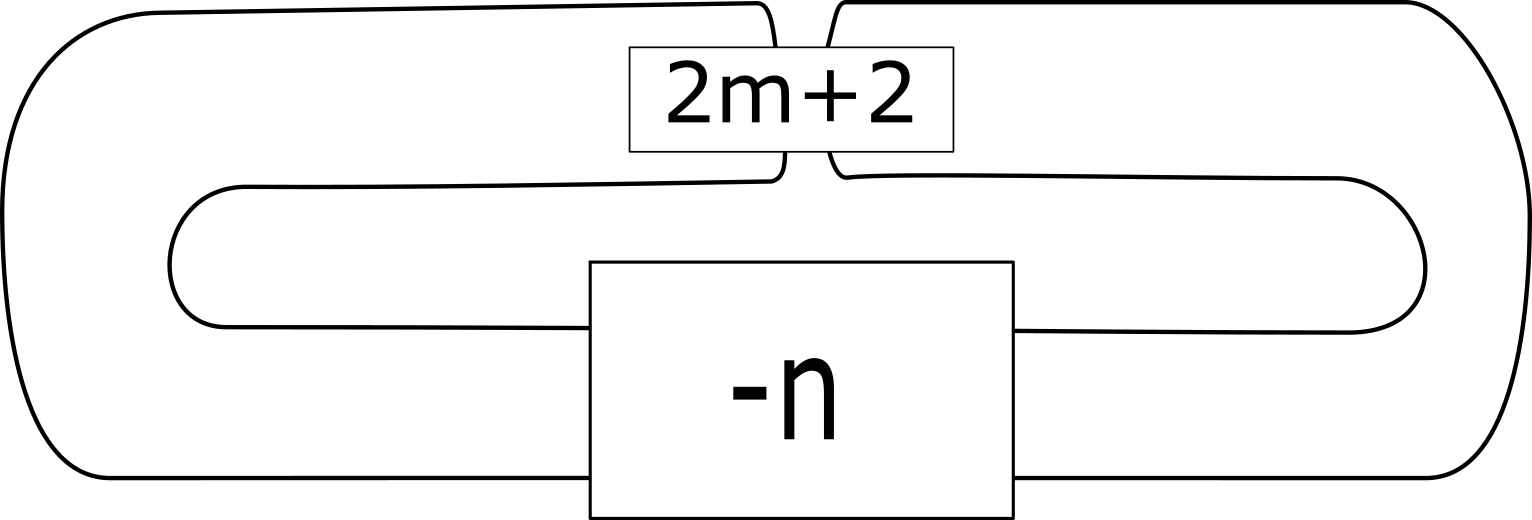}
    \caption{ The double twist knot $K(2m+2,-n)$, or in Conway notation $[-n,-2m-2]$. For $m=1$ $n=5$ this is $m(9,4)$; if $m=2$  $n=5$ we have $m(11a_{358})$;  if $m=1$ $n=7$ it is $m(11a_{342})$
    }
    \label{double twist knot}
\end{figure}

We remark that in \citep[Theorem 5.8]{OSct}  and \citep[Theorem 1.2]{Flt}  there are transverse non-simplicity statements about certain families of rational knots, and all the double twist in Theorem \ref{th 5.3} are included there. However, 
the knots $[-2m-1,-n,2]$ in Theorem \ref{th 5.2} are not. Moreover, instead of applying Theorem \ref{th 1.5} to individual green or blue circles, we can apply it to them simultaneously, or to the red circle on 3 arcs, or to some green, some blue and some red. Each of those gives different families of non-simple knots.

We can also apply our theorem to more complicated (non 2-bridge) non-simple knot. In \citep{NOTtk} Ng, Ozsv\'ath, and Thurston found many examples of non-simple knots. Those knots have pair of representatives $T_1$ and $T_2$ with same tb and rot but $\widehat{\theta}(T_1)=0$ and $\widehat{\theta}(T_2)\neq 0$, where $\widehat{\theta}$ is the Legendrian invariant living in grid homology  \citep{OSTltc}. It has been shown in \citep{BVVet} that $\widehat{\theta}$ is the same as $\mathfrak{\widehat{L}}$, so we can apply Theorem \ref{th 1.5} to those knots. For example the two knots in figure \ref{(2,3)(2,3)} are two Legendrian representatives of $(2,3)$ cable of $(2,3)$ torus knot with same tb and rot but different $\mathfrak{\widehat{L}}$ invariant, so if we apply Theorem \ref{th 1.5} to add twists in the circled region we will produce more non-simple knots. As we have seen, it's easy to produce a lot of infinite families of non-simple knots as long as we start with a non-simple ones that distinguished by $\mathfrak{\widehat{L}}$ invariants.

\begin{figure} [H]
    \centering
\includegraphics{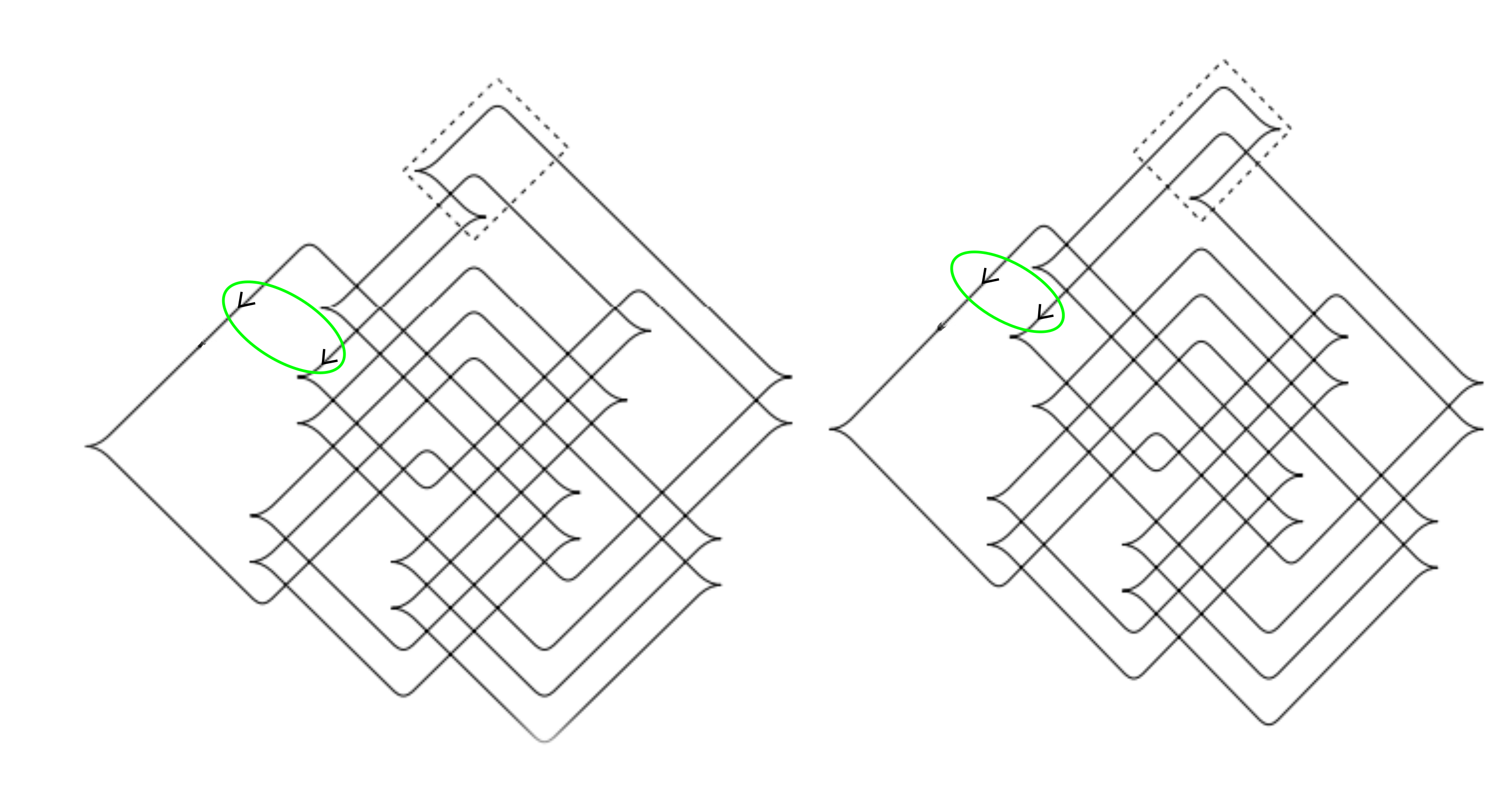}
    \caption{Legendrian fronts for $L_1$ (left) and $L_2$ (right), which are both
$(2, 3)$ cables of the $(2, 3)$ torus knot. They have same tb and rot ,but $\mathfrak{\widehat{L}}(L_1)=0$, $\mathfrak{\widehat{L}}(L_2)\neq 0$ (this diagram is from Figure 6 in \citep{OSTltc}; the dotted circle indicates the only region in which the diagrams differ).
    }
    \label{(2,3)(2,3)}
\end{figure}

\section{$Spin^c$ structure in the contact surgery corbordism} \label{Spin^c structure}

We start with the setting of $Y$ being a  rational homology sphere, $K$ an oriented rationally null-homologous knot with order $p$ which means $[K]$ in $H_1(Y,\mathbb{Z})$ has order $p$. Following the definition and convention in \citep{DingLiWuContact+1onrationalsphere}, a \textbf{rational Seifert surface} for smooth $K$ is a smooth map $i: F \rightarrow Y$ from a connected compact oriented surface
$F$ to $Y$ that is an embedding from the interior of $F$ into the exterior of  $L$, together with a  $p$-fold cover from
$\pa F$ to $K$. Let $N(K)$ be the closed neighborhood of $K$, and $\mu\subset \pa N(K)$ a meridian. We can assume that $i(F)\cap \pa N(K)$ consists of $c$ parallel
oriented simple closed curves such that each represents the same homology class $\nu$ in $H_1(\pa N(K), \mathbb{Z})$. Define the canonical longitude $\lambda_{can}$ be the longitude satisfying $\nu = t\lambda_{can} + r{\mu}$, where for simplicity the homology classes are also denoted by $\lambda_{can}$ and $\nu$ respectively, $t$ and $r$ are coprime integers with $0 \leq r < t$. In other words, $[\pa F]=c(t\lambda_{can} + r{\mu})$, and $\lambda_{can}$ is the choice of longitude for which $r/t$ is the unique representative of the rational self-linking number of $K$ in $[0, 1)$ (see also \citep{MTn} and \citep[Section 2.6]{RaouxTauinvariantsforknotsinrationalhomologysphere}).

If we perform smooth integer $n$ surgery along an order $p$ rationally null-homologous knot $K$, our convention is that the surgery coefficient is measured with respect to the canonical longitude $\lambda_{can}$, and we denote by $Y_n(K)$ the resulting manifold. Let $X_n(K)$ be the $4$-manifold
obtained by attaching a $4$-dimensional $2$-handle $H$ to $Y \times I$ along $K \times \{1\}$ where the coefficient $n$ is with respect to $\lambda_{can}$, in other words we have $\pa X_n(K) = (-Y) \cup Y_n(K)$. Let $C$  be the core of the $2$-handle in $X_n(K)$, with $\pa C = K \times \{1\}$. For a rational Seifert surface $i: F \rightarrow Y$ of $K$ we let $\Tilde{F}$ be the 2-cycle in $X_n(K)$ given by $\Tilde{F} = (i(F) \times \{1\}) \cup (-pC)$. We think of $\Tilde{F}$ as a ``capped off Seifert surface'', which is precisely true if $p=1$, i.e., the knot is null-homologous. Here $\Tilde{F}$ is merely a 2-chain; with more care one can construct a smooth representative of the class $[\Tilde{F}]\in H_2(X_n(K);\mathbb{Z})$ if desired.

Now we move on to an order $p$ Legendrian rationally null-homologous knot $L$ in some contact rational homology sphere $(Y,\xi)$. In \citep{BakerEtnyreRationallinking}, using rational Seifert surfaces Baker-Etnyre defined the rational Thurston-Bennequin number $tb_\mathbb{Q}(L)$ and rational rotation number $rot_\mathbb{Q}(L)$ for such a Legendrian knot $L$, and we refer the reader to \citep{BakerEtnyreRationallinking} for detailed definitions. Now, the Legendrian $L$ has a canonical framing $\lambda_c$ induced by the contact planes, which is given by $\lambda_c = \lambda_{can}+ k\mu$ for some $k$. In the null-homologous case $k$ is precisely the Thurston-Bennequin number, but in general we have the following. The definition in \citep{BakerEtnyreRationallinking} is that 
\[
tb_{\mathbb{Q}}(L):= \frac{1}{p} ([F] \cdot \lambda_c)= -\frac{1}{p} (p \lambda_{can}+cr\mu) \cdot  (\lambda_{can}+k\mu).
\]
(The sign appears because the intersection on the left happens in $\partial(Y - N(K))$ while on the right we work in $\partial N(K)$.) If we write $tb_{\mathbb{Q}}(L)= \frac{q}{p}$, then the above yields 
\begin{equation} \label{k(L) calculation}
    k=\frac{q+cr}{p}=tb_{\mathbb{Q}}(L)+\frac{cr}{p}.
\end{equation}

Observe that the number $k$ is uniquely determined by $L$, and we denote the corresponding $k$ as $k(L)$. Finally, note that performing $+n$-contact surgery on $L$ is given smoothly by $+n$ surgery with respect to contact framing $\lambda_c$, which is equivalent to doing $k+n$ surgery on $L$ with respect to the canonical framing $\lambda_{can}$.

We have the following naturality theorem of the contact invariant.

    \begin{theorem} \cite[Theorem 1.1] {MTn} \label{naturality of contact invariant}
    Let $L$ be an oriented rationally null-homologous Legendrian knot in a contact rational homology sphere $(Y,\xi)$ with non-vanishing contact invariant $c(\xi)$. Let $0<n \in \mathbb{Z}$ be the contact surgery coefficient, and $K$ be the smooth knot type of $L$. Let $W:Y \rightarrow Y_{k(L)+n}(K)$ be the corresponding rational surgery cobordism where $W=X_{k(L)+n}(K)$, and consider $\xi_{n}^-(L)$ on $Y_{k(L)+n}(K)$. There exist a $Spin^c$ structure $\mathfrak{s}$ on $-W$ such that the homomorphism 
        
        $$F_{-W,\mathfrak{s}}: \HFa(-Y) \rightarrow \HFa(-Y_{k(L)+n}(K))$$ 
        induced by $W$ with its orientation reversed satisfies 
        $$F_{-W,\mathfrak{s}}(c(\xi))= c(\xi_n^-(L)).$$
\end{theorem}

\begin{remark}
    In \citep{MTn} the knot is assumed to be null-homologous, but the version stated above follows by the same arguments. In fact, even though we cite the above theorem from \citep{MTn}, naturality under $+n$ contact surgery actually comes from the combination of \citep[Theorem 1.2]{Bc} (naturality of the contact invariant under capping off cobordism) and Theorem \ref{3.4} (Equivalence of capping off and $+n$-contact surgery). Since we don't require the knot to be null-homologous in those theorems we have the rationally null-homologous version for the contact surgery. 
\end{remark}

What we want to characterize here is the $Spin^c$ structure mentioned in the above theorem without conjugation ambiguity. Note that in the following statement we use $y$ to denote the order of $L$ in homology instead of $p$, since elsewhere $p$ refers to the order of the induced knot in $Y_n(L)=Y_{k(L)+n}(K)$.

\begin{theorem} \label{thm: spin^c for RHS}
    In the above setting, assume $Y_n(L)$ is also a rational homology sphere. Then the $\mathfrak{s}$ in Theorem \ref{naturality of contact invariant} has the property that $$ \langle c_1(\mathfrak{s}),[\Tilde{F}] \rangle = 
      y(rot_\mathbb{Q}(L)+n-1) $$
 where $y$ is the order of $[L]$, $F$ is a rational Seifert surface for $L$ and $\Tilde{F}$ is the ``capped off" surface of $F$.
\end{theorem}

Then Proposition \ref{prop 1.4} follows directly from the above theorem where $L$ is null-homologous in $Y$. To prove Theorem \ref{thm: spin^c for RHS} we again need to first prove the analogue theorem for the capping off cobordism. 

\begin{theorem} \label{capping off spinc}
Let $(P_{g,r},\phi)$ be an abstract open book with genus $g$ and $r>1$ binding components with a chosen transverse binding component $B$.Then capping off $B$ we get a new open book $(P_{g,r-1},\phi')$. Denote by $(M,\xi)$, $(M',\xi')$ the contact 3 manifolds corresponding to those two open books. The capping off cobordism gives rise to a map 
\begin{equation}
    F_{B,\mathfrak{s}}: \HFa (-M') \rightarrow \HFa (-M)
\end{equation}
where $\mathfrak{s}$ is a $Spin^c$ structure on the cobordism $W$ from $-M'$ to $-M$. Then
\begin{enumerate} [(i)]
    \item \citep[Theorem 1.2]{Bc} if $M'$ is a rational homology sphere, there is a choice of \s \ for which 
\begin{equation}
    F_{B,\mathfrak{s}}(c(\xi'))=c(\xi)
\end{equation}
holds.
\item  Assume also that $M$ is a rational homology sphere, and let $L$ be the Legendrian push off of $B$, realised as a curve on the page of the open book $(S_{g,r},\phi)$ that is parallel to the binding $B$. Then the $Spin^c$ structure \s mentioned in $(i)$ satisfies: $$ \langle c_1(\mathfrak{s}),[\Tilde{V}] \rangle = -p(rot_\mathbb{Q}(L)+1),$$
where $p$ is the order of $[B]$, and $\Tilde{V}$ is homology class in $W$ represented by the ``capped off'' rational Seifert surface $V$ of $B$. 
\end{enumerate}
\end{theorem}

To prove the second part of the above theorem we first need the following (rationally null-homologous version of) Lemmas from \citep[section 4]{OSct}. 

\begin{lemma} \label{lem: domain of L}
    Let $(P,\phi)$ be an abstract open book and $M(\phi)$ the corresponding $3$-manifold. Let $L$ be a homologically non-trivial closed curve on $P$, then $p[L]\in H_1(P)$ is in the kernel of $H_1(P)\rightarrow H_1(M(\phi))$ if and only if it can be written as $p[L] =\phi_*(Z)-Z$ for some $Z\in H_1(P,\pa P)$.
\end{lemma}

\begin{proof}
    The case for the connected binding is proved in \citep[Lemma 4.2]{OSct}, where in the connected binding case $Z$ is actually an absolute class in $H_1(P)$. Now we suppose $(P,\phi)$ has two binding components. To make the proof clear we first pick useful basis representatives for $H_1(P)$ as follows. We let $l_2$ be a curve that is parallel to one of the boundary components of $P$, then we pick curves $l_i$ for $3\leq i\leq k$ such that they are disjoint from $l_2$ and form a standard symplectic basis for $H_1(P/l_2)$. Then it is clear that $\{[l_2],[l_3],...,[l_k]\}$ form the basis of $H_1(P)$.
    
    Now we positively stabilize the open book by attaching a $1$-handle between the two binding components to obtain a new page $P^+$. Then $(P^+,\phi^+)=(P^+, \phi \circ \tau_c)$ is an open book with connected binding where $\tau_c$ is a right handed Dehn twist along a curve $c$ in $P^+$ that intersects both the co-core of the $1$-handle and $l_2$ exactly one time but misses all the other $l_i$. Then from the one-boundary case we know $p[L]=\phi_*^+(Z^+)-Z^+$ for some $Z^+\in H_1(P^+)$. We view $P$ as a subsurface of $P^+$, and let $l_1=c$ where the orientation is chosen so that $l_1 \cdot l_2=1$; then $\{[l_1],[l_2],...,[l_k]\}$ is a basis for $H_1(P^+)$. First we have the following observations regarding this basis: \begin{enumerate}
        \item If we let $\overline{l_1}=l_1 \cap P$ then $\overline{l_1} \in H_1(P,\pa P)$, and moreover $\phi^+_*(l_1)-l_1=\phi_*(\overline{l_1})-\overline{l_1}$ is an absolute class on $H_1(P)$.
        \item $\phi^+_*(l_2)-l_2=c=l_1$
        \item $\phi^+_*(l_i)-l_i=\phi_*(l_i)-l_i$  which is also an absolute class on $H_1(P)$, for $i\neq 1, 2$. 
        \item For any absolute class $[O]\in H_1(P)$, $O \cdot l_2=0$.
    \end{enumerate}

Using the above basis we write $Z^+=\sum_{i=1} u_i l_i$ and $\phi^+_*(Z^+)-Z^+=\sum_{i=1}v_il_i$. We can also express $\phi^+_*(Z^+)-Z^+$ as $\sum_{i=1}u_i(\phi^+_*(l_i)-l_i)$, and we claim $v_1=u_2$. This is because the coefficient $v_1$ of $l_1$ is the same as the intersection number between $\phi^+_*(Z^+)-Z^+=\sum_{i=1}u_i(\phi^+_*(l_i)-l_i)$ and $l_2$. The observations above imply $(\phi^+_*(l_2)-l_2) \cdot l_2=1$, and $(\phi^+_*(l_i)-l_i) \cdot l_2=0$ for $i\neq 2$. Thus $(\phi^+_*(Z^+)-Z^+) \cdot l_2 =u_2$ which proves the claim.

Our assumptions tell us $L\subset P$ which implies $[L] \cdot l_2 =0$, and we also know that $p[L]=\phi_*^+(Z^+)-Z^+$, so $(\phi_*^+(Z^+)-Z^+) \cdot l_2=0$ thus we conclude $u_2=0$. Using the observation $1$ and $3$, we can express $p[L]=\phi_*(Z)-Z$, where $Z=u_1\overline{l_1}+\sum_{i=3} u_il_i \in H_1(P,\pa P)$. 

For open book with more than two binding components the above argument easily extends inductively.
\end{proof}

\begin{lemma} \citep[Lemma 4.4]{OSct} \label{lem: rotation and euler measure}
Let $L\in M (\phi)$ be an order $p$ rationally null-homologous Legendrian knot supported in the page $P$ of the open book (page framing equals to contact framing). Let $\mathfrak{p}$ be a two-chain with $\pa \mathfrak{p}= p[L]+(Z-\phi_*(Z))$ for some one–cycle $Z\in H_1(P,\pa P)$. Then, $p \cdot rot_{\mathbb{Q}}(L)$ is equal to $e(\mathfrak{p})$, the Euler measure of $\mathfrak{p}$ (see \cite[section 7.1]{OzsvathSzaboHDandthreemanifoldinvariants} for details about Euler measure).
\end{lemma}

\begin{proof}
    This is more straight forward than the previous lemma. We start with the same set up as above let $(P,\phi)$ be the open book with multiple binding components and $(P',\phi')$ be the positive stabilization of $(P,\phi)$ with connected binding (note positive stabilization does not change the contact structure, so when we view $P$ as subsurface of $P'$ the Legendrian knot $L$ sitting on $P'$ is Legendrian isotopic to the one sitting on $P$). \citep[Lemma 4.4]{OSct} says any two chain $\mathfrak{p'}$ with $\pa \mathfrak{p'}= p[L]+(Z'-\phi_*(Z'))$, where $Z'\in H_1(P')$ is the cycle in the proof of the above lemma, satisfies $e(\mathfrak{p'})=p \cdot rot(L)$.

    Then according to the proof of the above lemma it's clear that the corresponding two chain $\mathfrak{p}=p[L]+(Z-\phi_*(Z))$, where $Z\in H_1(P, \pa P)$ as described above, gives rise to such a $\mathfrak{p'}$ when we include $\mathfrak{p}$ into $P'$. Thus the lemma follows.
\end{proof}

For a rationally null-homologous Legendrian $L$ on page $P$ with order $p$, we Let $\{a_1, ..., a_k \}$ be a basis for $H_1(P, \pa P)$ with the property that $a_i$ are embedded \textbf{arcs} in $P$, $a_2, ..., a_k$
are disjoint from $L$ and $a_1$ meets $L$ in a single transverse intersection point with the convention $L \cdot a_1=+1$. For $Z\in H_1(P,\pa P)$ and two-chain $\mathfrak{p}$ we found in the above two lemmas, using above basis we can rewrite them as  \[Z=\sum_{i=1}^k n_i \cdot a_i\text{, and } \pa \mathfrak{p} = p[L] + \sum_{i=1}^k n_i \cdot (a_i-\phi_*(a_i)).\] 


Now we are ready to proceed the proof of Theorem \ref{capping off spinc}. We first recall the setting of the Heegaard triple diagram $(\Sigma, \alpha, \gamma, \beta, z)$ that describe the capping off cobordism from $M$ to $M'$, for simplicity we let $P=P_{g,r}$, and $P'=P_{g,r-1}$. 

$\Sigma$ is the Heegard surface of the union of two pages $P_{+1} \cup -P_{-1}$. Let arcs $\{a_1, ..., a_k \}$ be a basis for $P_{+1}$ with the property that $a_2, ..., a_k$
are disjoint from $L$ and $a_1$ meets $L$ in a single transverse intersection point (note this is the exactly same basis we used for $Z$ and $\mathfrak{p}$), where $L$ is the Legendrian push off of the binding $B$ on $P_{+1}$. We let $c_i$ be a push off of $a_i$ for all $i$, $b_i$ be a further push off of $c_i$ for $i\neq 1$, and $b_1$ be the parallel push of of binding $B$ on $P_{+1}$. Then in particular we will have a triangle $\Delta_i$ 
formed by $a_i$, $c_i$ and $b_i$ for all $i$ (see Figure \ref{Oriented triple}). 

Now we let $\alpha_i=a_i \cup \overline{a_i}$ and $\gamma_i=c_i\cup \overline{\phi(c_i)}$ for all $i$. For the $\beta$ curves let $\beta_i=b_i\cup \overline{\phi(b_i)}$ for $i\neq 1$, and $\beta_1=b_1$. For base point $z$, we put it outside the thin strips region between the arcs. Then the manifolds $M$ and $M'$ are represented by $(\Sigma,\alpha,\gamma)$ and $(\Sigma,\alpha,\beta)$ respectively. The pointed Heegaard triple $(\Sigma,\alpha,\gamma,\beta,z)$ describes the capping off cobordism from $M$ to $M'$, and $(\Sigma, \gamma,\beta, \alpha,z)$ is the opposite cobordism from $-M'$ to $-M$. This is exactly the same setting in section 4 but ignoring the extra Legendrian knot. 

According to \cite{Bc}, the small triangle we formed by $a_i$, $c_i$, $b_i$ is representing the $Spin^c$ structure in Theorem \ref{capping off spinc} ($i$), since $(\Sigma,\alpha,\gamma,\beta,z)$ and $(\Sigma, \gamma,\beta, \alpha,z)$ represent the same $4$-manifold $W$ we will calculate this $Spin^c$ structures using $(\Sigma,\alpha,\gamma,\beta,z)$. 

To calculate the first Chern class we will use the formula in \cite[Proposition 6.3]{OShi}. If we let $\psi: \Delta \rightarrow Sym^k(\Sigma)$ be the Whitney triangle correspond to the domain consists of little triangles  $\Delta$ we are interested in, $D$ a triply periodic domain representing the two-dimensional homology class $H(D)  \in H_2(W,\mathbb{Q})$, then \begin{equation} \label{c_1 of triangle}
        \langle c_1(\mathfrak{s}_z(\psi),H(D) \rangle = e(D)+\#(\pa D)-2n_z(D)+2\sigma(\psi,D),
    \end{equation}
where $e(D)$ is the Euler measure of $D$, $\#(\pa D)$ is the coefficient sum of all terms in $\pa D$, and $\sigma(\psi,D)$ is the dual spider number. The dual spider number $\sigma(\psi,D)$ can be calculated as follows: 

We first choose an orientation on $\alpha$, $\gamma$ and $\beta$\, we let $\alpha'$, $\gamma'$ and $\beta'$ be the leftward push offs of the corresponding curve. Let $\pa_{\alpha'}(D)$, $\pa_{\gamma'}(D)$ and $\pa_{\beta'}(D)$ be the $1$-chains obtained by translating the corresponding components of $\pa D$. Let $u$ be an interior point of $\Delta$ so that $\psi(u)$ misses $\alpha$, $\gamma$, $\beta$ curves, then choose
three oriented paths $r$, $t$ and $s$, from $u$ to the  $\alpha$, $\gamma$, $\beta$ boundaries respectively such that $r$, $t$ and $s$ are in the 2-simplex $\Delta$ that is the domain of $\psi: \Delta \rightarrow Sym^k(\Sigma)$. Identifying these arcs with their image $1$-chain in $\Sigma$, the dual spider number is given by
\begin{equation}
    \sigma(\psi,D)=n_{\psi(u)}(D)+\pa_{\alpha'}(D) \cdot r+ \pa_{\gamma'}(D)\cdot t + \pa_{\beta'}(D)\cdot s
\end{equation}

Now we are ready to prove the  Theorem \ref{capping off spinc} ($ii$)

\begin{proof} [Proof of Theorem \ref{capping off spinc} ($ii$)]
We first need to identify a triply periodic domain in the triple diagram. 

Recall that we have a domain $\mathfrak p$ in the page $P$ with $\pa \mathfrak{p} = p[L] + \sum_{i=1}^k n_i \cdot (a_i-\phi_*(a_i))$ and $e(\mathfrak{p})= p\cdot rot_{\mathbb{Q}}(L)$. Consider $\mathfrak p$ as lying on $P_{-1}$, and $P_{-1}$ lying in $\Sigma=P_{+1}\cup -P_{-1}$. With this point of view we write $\poverline$ instead as the domain ${\mathfrak{p}}$ on $\Sigma$ with
\[\pa \overline{\mathfrak{p}}= p[\overline{L}]+\sum_{i=1}^k n_i \cdot (\overline{a_i}-\phi_*(\overline{a_i})), \] where we are using $\overline{L}$ and $\overline{a_i}$ to mean the images of $L$ and $a_i$ (considered on $P_{+1}$) under identity map on the opposite page $-P_{-1}$. Then $\overline{\mathfrak{p}}$ satisfies 
$e(\overline{\mathfrak{p}})=-p\cdot rot_{\mathbb{Q}}(L)$. 

We then observe that the cycle $(\alpha_i-\gamma_i)$ is exactly $ (\overline{a_i}-\phi_*(\overline{a_i}))$, and thus when we push $\overline{L}$ across $B$ to $L$ (from $-P_{-1}$ to $P_{+1}$) we obtain a corresponding domain (still denoted $\overline{\mathfrak{p}}$) on $\Sigma$ with 
$$\pa \overline{\mathfrak{p}}= p[L]+\sum_{i=1}^k n_i \cdot (\alpha_i-\gamma_i) $$ such that $e(\overline{\mathfrak{p}})=-p\cdot rot_{\mathbb{Q}}(L)$. 

Now since $L$ is the Legendrian push off of the binding $B$ its orientation coincides with that of $B$, which is compatible with the orientation of the page (the page $P_{+1}$ is oriented counter clockwise). Thus the arc $a_1$ is oriented by the requirement $L.a_1 = +1$, which induces natural orientations on  $\alpha_1$ and $\gamma_1$. We manually give an orientation to $\beta_1$ that is opposite to the orientation of $L$, and for the orientation of the rest curves we make choice such that near each little triangle $\Delta_i$ for $i\neq 1$, it looks like $\Delta_2$ as shown in Figure \ref{Oriented triple}. We will use those orientation for calculation. 

\begin{figure} 
    \centering
\includegraphics[width=\textwidth]{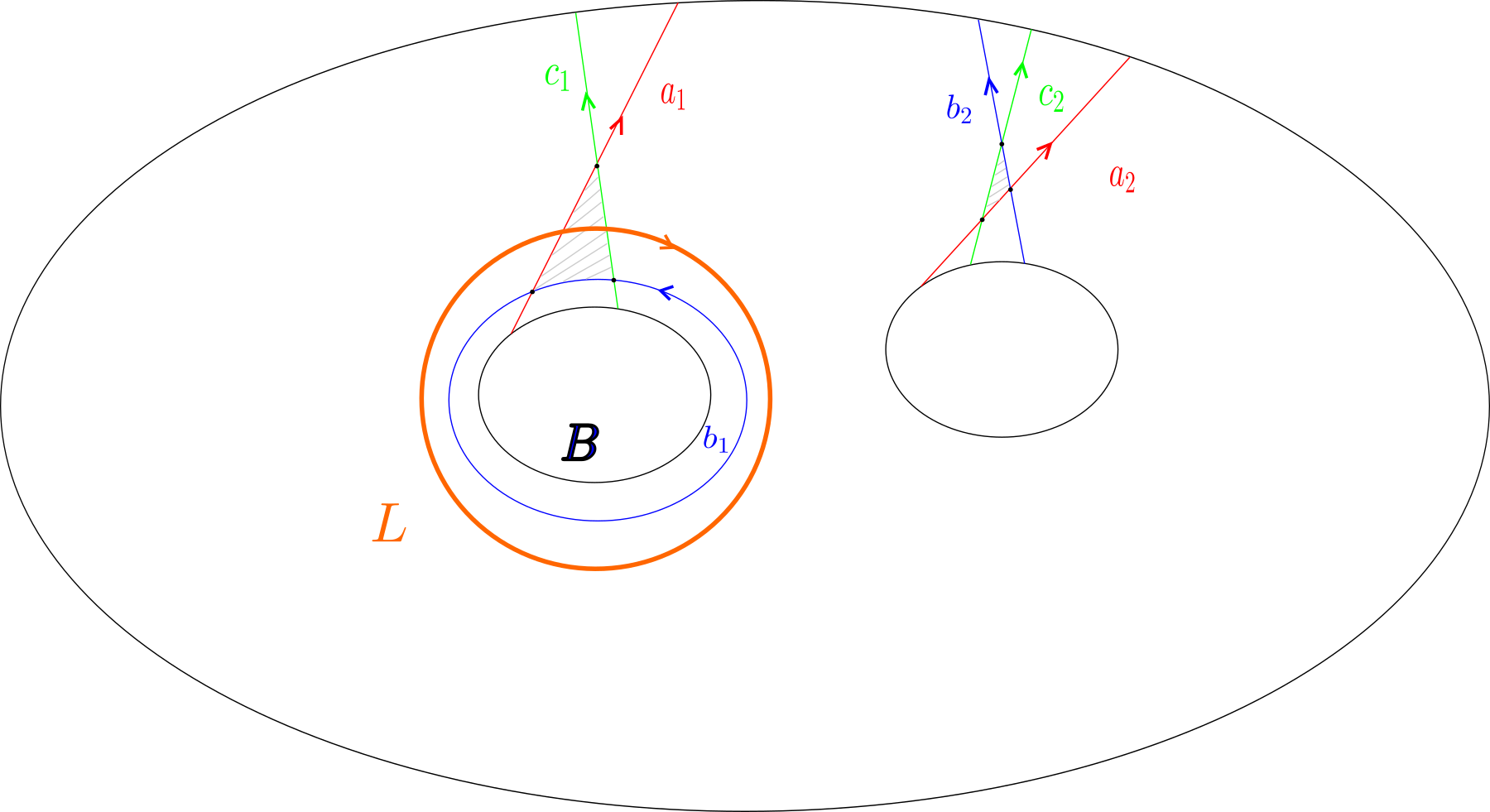}
    \caption{We oriented the curves as it shown, and the $Spin^c$ structure \s  
     \ corresponds to the shaded small triangle.
    }
    \label{Oriented triple}
\end{figure}

Notice that $L\simeq -\beta_1$, thus we can take our $\overline{\mathfrak{p}}$ to be a triply periodic domain on $(\Sigma,\alpha,\gamma,\beta,z)$ with boundary $-p\beta_1+\sum_{i=1}^k n_i \cdot (\alpha_i-\gamma_i)$. Moreover we can add or subtract a multiple of the whole $\Sigma$ to make $n_z(\mathfrak{\overline{p}})=0$. Thus by formula \ref{c_1 of triangle}

\begin{align*}
    \langle c_1(\mathfrak{s}_z(\psi),H(\poverline) \rangle &= e(\poverline)+\#(\pa \poverline)-2n_z(\poverline)+2\sigma(\psi,\poverline) \\
    &= -p\cdot \rotq(L) + (-p+\sum n_i-n_i) - 0
    + 2\sigma(\psi,\poverline)\\
        &= -p\cdot \rotq(L) -p + 2\sigma(\psi,\poverline)
\end{align*}    

We claim that $2\sigma(\psi,\poverline)=0$. We draw the ``dual spider" $\alpha'$, $\gamma'$, $\beta'$ and $r$, $s$, $t$ based on the orientation, recall $\Delta=\Delta_1+\Delta_2+...+\Delta_k$, and except $\Delta_1$ the neighborhood of the rest triangles look the same, so the contribute for the dual spider can be divided into two case for $\Delta_1$ and $\Delta_i$ where $i\neq 1$. 

We first look at the contribution of dual spider number $\sigma(\psi,\poverline|)_{\Delta_1}$from $\Delta_1$, see Figure \ref{fig: dual spider}. Since we assume $n_z(\poverline)=0$, and the multiplicity of $\alpha_1$ and $\gamma_1$ are $n_1$ and $-n_1$ respectively, this force $n_{\psi(u)}(\poverline)|_{\Delta_1}=-n_1$. It is easy to see $\pa_{\alpha'_1}(\poverline) \cdot r =0 $ and $\pa_{\beta'_1}(\poverline) \cdot s=0$ because both $\alpha'_1$ and $\beta'_1$ are outside of the triangle $\Delta_1$. The last quantity $\pa_{\gamma'_1}(\poverline) \cdot t = n_1$ because $\gamma'_1 \cdot t =-1$ and $\gamma_1$ has multiplicity $-n_1$. Thus 
\begin{align*}
    \sigma(\psi,\poverline)|_{\Delta_1}&=n_{\psi(u)}(\poverline)|_{\Delta_1}+\pa_{\alpha_1'}(\poverline) \cdot r+ \pa_{\gamma_1'}(\poverline)\cdot t + \pa_{\beta_1'}(\poverline)\cdot s \\
    &=-n_1+0+n_1+0.\\
    &=0
\end{align*} 

We then look at the contribution of dual spider number $\sigma(\psi,\poverline|)_{\Delta_i}$ from $\Delta_i$ for $i\neq 1$, see Figure \ref{fig: dual spider}. Again the assumption $n_z(\poverline)=0$ with the fact that the multiplicity of $\alpha_i$ and $\gamma_i$ are $n_i$ and $-n_i$ respectively force $n_{\psi(u)}(\poverline)|_{\Delta_i}=n_i$. This time we see $\pa_{\gamma'_i}(\poverline) \cdot t =0 $ because $\gamma'_i$ are outside of the triangle $\Delta_i$, and $\pa_{\beta'_1}(\poverline) \cdot s=0$ because $\beta_i$ has multiplicity 0. The last quantity $\pa_{\alpha'_1}(\poverline) \cdot r = -n_1$ because $\alpha'_i \cdot t =-1$ and $\alpha_i$ has multiplicity $n_1$. Thus 
\begin{align*}
    \sigma(\psi,\poverline)|_{\Delta_i}&=n_{\psi(u)}(\poverline)|_{\Delta_i}+\pa_{\alpha_i'}(\poverline) \cdot r+ \pa_{\gamma_i'}(\poverline)\cdot t + \pa_{\beta_i'}(\poverline)\cdot s \\
    &=n_i-n_i+0+0\\
    &=0
\end{align*}

Thus each triangles contribute $0$, we have $\sigma(\psi,\poverline)=0$, which proves the claim. Hence when we plug in back to the Chern class evaluation we get 
$$\langle c_1(\mathfrak{s}_z(\psi),H(\poverline) \rangle = -p\cdot \rotq(L) -p $$

To complete the proof we claim that  $H(\poverline)$ is the same as the class of the capped off rational Seifert surface. This is clear by the fact that the $\beta$-boundary of $\poverline$ is just $pL$, together with the construction of the identification between periodic domains in $(\Sigma,\alpha,\gamma,\beta)$ and homology classes in the cobordism between $M$ and $M'$ \citep[Proposition 8.2]{OSht}.

\end{proof} 

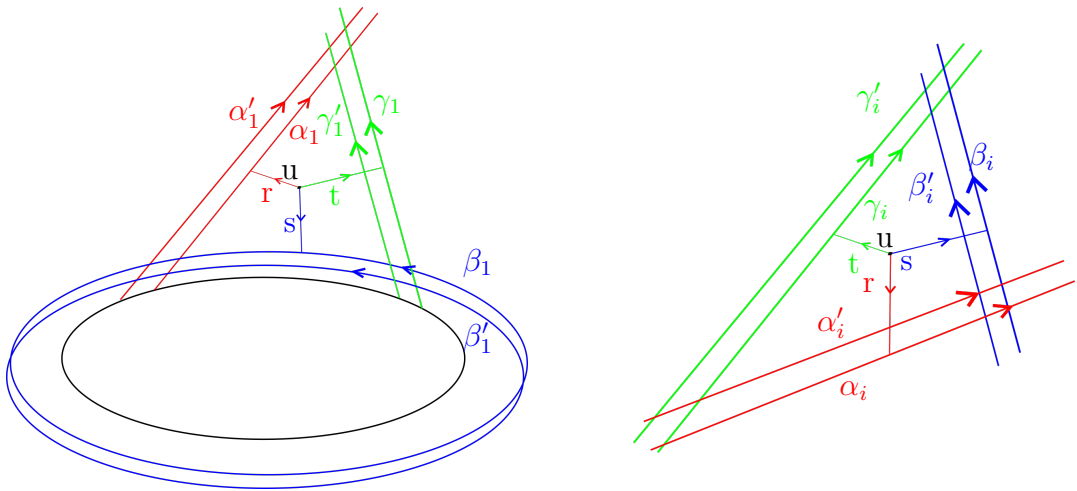
\begin{figure}[htb!]
\centering
\begin{tikzpicture}
    \node at (-4,0){\includegraphics[scale=0.06]{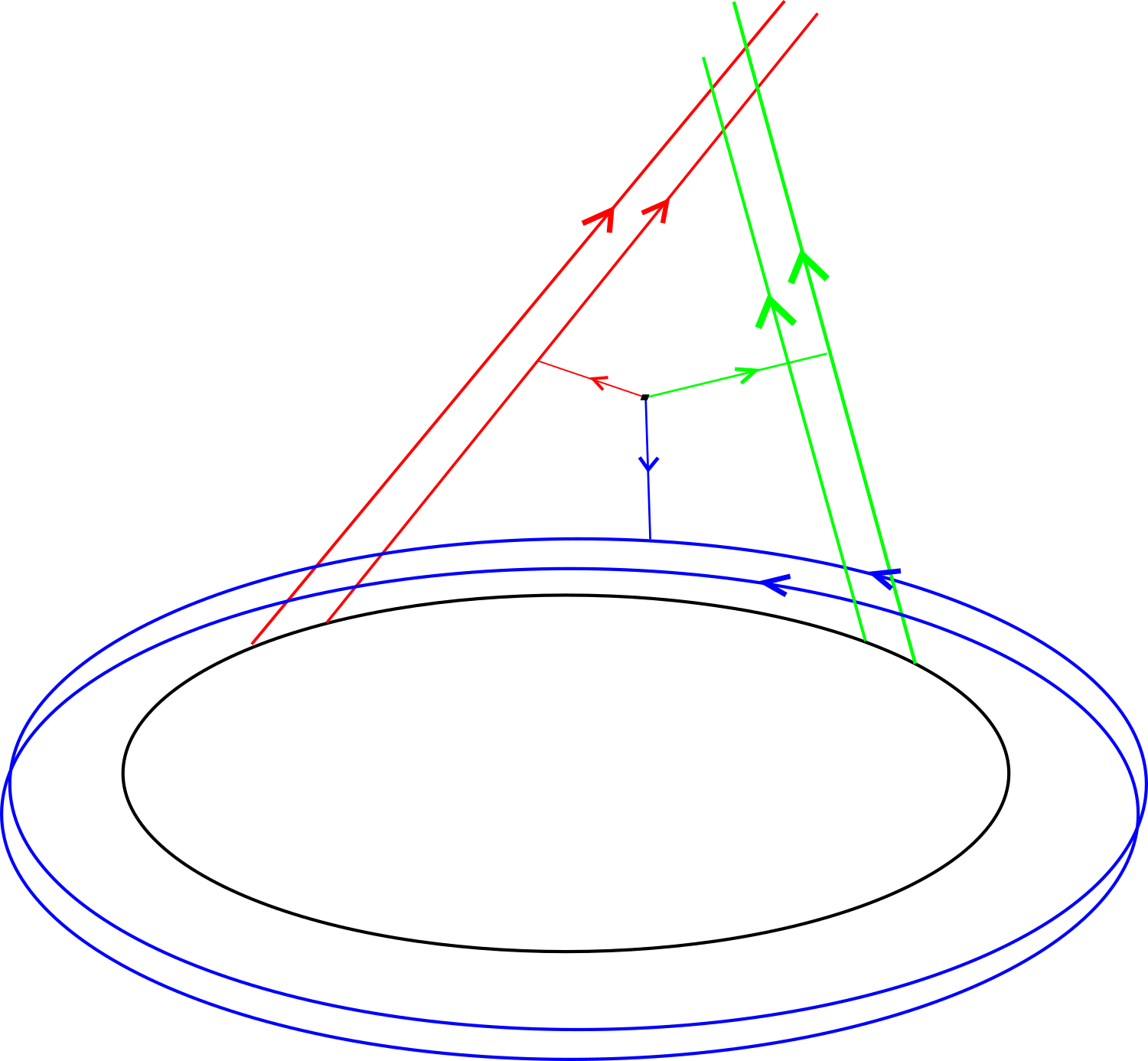}};
     \node at (-3.7,1){u};
      \node [text=green] at (-3.1,0.7){t};
      \node [text=red] at (-4,0.7){r};
       \node [text=blue] at (-3.7,0.3){s};
         \node [text=green] at (-3.15,1.7){$\gamma_1'$};
         \node [text=green] at (-2.4,1.9){$\gamma_1$};
       \node [text=red] at (-3.5,1.5){$\alpha_1$};
        \node [text=red] at (-4.3, 1.8){$\alpha_1'$};
         \node [text=blue] at (-1.2,-1.2){$\beta_1'$};
         \node [text=blue] at (-1.2,-0.2){$\beta_1$};

         \node at (3.8,0){\includegraphics[scale=0.07]{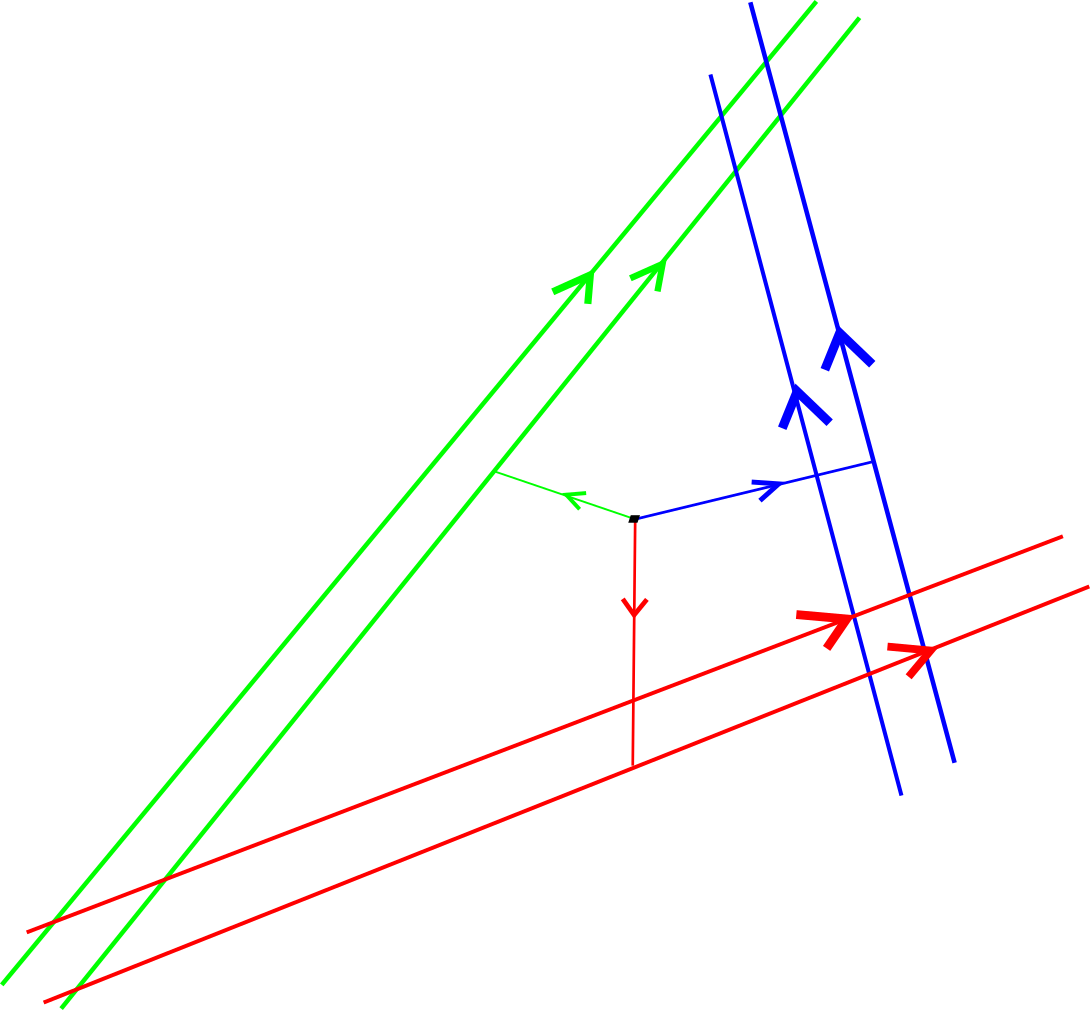}};
     \node at (4.21,0.1){u};
      \node [text=green] at (3.8,-0.2){t};
      \node [text=red] at (4,-0.5){r};
       \node [text=blue] at (4.5,-0.2){s};
         \node [text=green] at (4.1,0.5){$\gamma_i$};
         \node [text=green] at (4,2){$\gamma_i'$};
       \node [text=red] at (3.5,-1){$\alpha_i'$};
        \node [text=red] at (3.8,-1.9){$\alpha_i$};
         \node [text=blue] at (4.7,0.8){$\beta_i'$};
         \node [text=blue] at (5.5,1.2){$\beta_i$};
\end{tikzpicture}
    \caption{The picture on the left describes the dual spider near $\Delta_1$, and the picture on the right describes the dual spider near $\Delta_i$ for $i\neq 1$.}
    \label{fig: dual spider}
\end{figure}


Before going in to the proof of Theorem \ref{thm: spin^c for RHS}, we first recall some the general settings about the $+n$ contact surgery, capping off and rationally null-homologous knot.

We will use the same notation as is described in section \ref{Contact surgery and Capping off cobordism}. Namely, $(P',k\circ \phi)$ is a stabilized open book compatible with $(Y,\xi)$ with $L$ on its page, and the stabilization $L^-$ parallel to some binding $B$; $(P',\phi')$ is an open book compatible with $(Y_n(L),\xi_n^-(L))$ (where $\phi'$ is obtained by composing $k\circ\phi$ with some twists along $L$ and $L^-$); and $B_L$ is the binding in $(P',\phi')$ corresponding to $B$ in $(P',k\circ \phi)$. Thus the cobordism corresponding to $+n$ contact surgery on $L$ in $Y$ is the same 4-manifold (with opposite orientation) as the  cobordism corresponding to capping off $B_L$ in $Y_n(L)$. Observe that the capped-off rational Seifert surfaces $\Tilde{F}$ (for $L$) and $\Tilde{V}$ (for $B_L$) are both generators of the second rational homology of this cobordism, but need not be identical classes.

We also recall that $[\pa F]=c(t\lambda_{can} + r{\mu})$, where $\lambda_{can}$ is the canonical framing and $ct=y$ is the order of $L$. Moreover if $\tbq(L)=\frac{x}{y}$, then contact $+n$ surgery on $L$ is the same as smooth $k(L)+n$ surgery on $L$ which is equivalent to smooth $\frac{x+cr+ny}{y}$ (according to \ref{k(L) calculation}) surgery on $L$.

Under the above setting we are able to state the most essential Lemma we need to prove the Theorem \ref{thm: spin^c for RHS}. 

\begin{lemma} As classes in $H_2(W;\mathbb{Z})/Tors$, we have \label{lem: homology class depend on smooth coefficient}
    \begin{equation*} 
    [\Tilde{F}]= \begin{cases} 
      -[\Tilde{V}] & \text{ if } x+ny>0\\
      [\Tilde{V}] & \text{ if } x+ny<0 
      \end{cases}.
\end{equation*}
\end{lemma}

\begin{proof}
    First notice the classes $[\Tilde{F}]$ and $[\Tilde{V}]$ are represented by  2-chains obtained by adding $y$ (the order of $L$ in homology) copies of the core disk $C$ of the surgery handle to $F$ in the first case, or $p$ (the order of $B_L$ in homology) copies of the cocore  disk $C_c$ to $V$ in the second case. Here the disk $C$ is oriented so that $\partial C = -L$, and we take the cocore disk to be oriented so that its signed intersection with the core disk is $+1$. It follows that $[\Tilde{F}]\cdot [\Tilde{V}] = \pm yp$, where the sign depends on the relative orientation between $\partial V$ (or equivalently $L'$) and the boundary of the cocore. We make the following $2$ claims. 
    
    Claim $1$: $[\Tilde{F}]\cdot [\Tilde{V}] = - yp$
    
    Claim $2$: $[\Tilde{F}]\cdot [\Tilde{F}] = y(x+ny)$\\
Once we have achieved the above two claims then by  \cite[commment after Lemma 5.1]{MTn} or \citep[Lemma 5.2]{DingLiWuContact+1onrationalsphere} the order of $L'$ is $p=|c(k(L)+n)t-cr|=|c(\frac{x+cr+ny}{y})t-cr|=|x+ny|$. Thus Claim 1 implies $[\Tilde{F}]\cdot [\Tilde{V}] = -y|x+ny|$. But $[\Tilde{F}]$ and $[\Tilde{V}]$ rational generators of $H_2(W)$, hence (modulo torsion) one is a multiple of the other; on the other hand we now have that $[\Tilde{F}]\cdot [\Tilde{F}] = \pm [\Tilde{F}]\cdot [\Tilde{V}]$ where the sign depends on the positivity of $x+ny$. The lemma follows.

Now we will first prove Claim $1$. We denote $C_c(L')$ to be the oriented cocore of the handle correspond to the orientation of $L'$, then the Claim $1$ is equivalent to $C \cdot C_c(L') =-1$. 

 To figure out the intersection number we recall that the DGS alogrithm (Theorem \ref{3.2}, c.f. \cite{DGS}) says that $+n$ surgery on $L$ is the same as  surgery along a link $L_1 = L, L_2, \ldots, L_n$, so that $L_2$ is the negativly stabilized Legendrian push off of $L_1$, and  $L_i$ is the Legendrian push off of $L_{i-1}$ for $i = 3,\ldots, n$. Moreover, if we denote $L_0$ be a further push off of $L_n$ then $L'$ corresponds to $L_0$ after performing contact $+1$ surgery on $L_1$ and contact $-1$ surgery on $L_2,\ldots, L_n$. (Remark here the notation is a bit different from the notation in Theorem \ref{3.2}.) 

We will first slide every $L_i$ for $i = 2,\ldots, n$ with smooth framing $k(L)-2$ over $L_1$ in order and further slide $L_0$ over $L_1$. If we denote $s(L_i)$ to be the new knots in the surgery diagram corresponding to $L_i$ after the slide, then it's not hard to see that for $i = 2,\ldots, n$ $s(L_i)$ are all isotopic to the meridian of $L$ with smooth framing $-1$, and $s(L_0)$ is also isotopic to the meridian of $L$. Thus after we blow down all the $s(L_i)$ for $i = 2,\ldots, n$, we are back to smooth $k(L)+n$ surgery on $L_1=L$, and $s(L_0)$ is still the meridian of $L_1$ that corresponds to the cocore of the handle. 

Now $[C] \cdot [C_c(L')]$ is the same as  $lk_{\mathbb{Q}}(L_1,s(L_0))$. Equipping $L_1$ and $L_0$ with the same orientation at the beginning, the slide of $L_0$ over $L_1$ becomes a handle subtraction, so that $lk_{\mathbb{Q}}(L_1,s(L_0))=lk_{\mathbb{Q}}(L_1,L_0)-lk_{\mathbb{Q}}(L_1,L_1)=\tbq(L_1)-(\tbq(L_1)+1)=-1=[C] \cdot [C_c(L')]$, which finishes the proof of Claim $1$.

Claim $2$ is more straight forward following \citep[Lemma 5.1]{MTn} that tells us the self-intersection $[\Tilde{F}]\cdot [\Tilde{F}]=y((k(L)+n)y-cr)=y(\frac{x+cr+ny}{y}y-cr)=y(x+ny)$, which completes the proof.
\end{proof}

\begin{proof}[Proof of Theorem \ref{thm: spin^c for RHS}]

First from Theorem \ref{3.4} and \ref{capping off spinc} we have that the  $Spin^c$ structure we are interested in satisfies 
\[
\langle c_1(\mathfrak{s}), [\Tilde{V}]\rangle = - p(rot_{\mathbb{Q}}(L')+1).
\]
Combined with Lemma \ref{lem: homology class depend on smooth coefficient} this means
 
\begin{equation}
\label{sign of c1 depend on smooth coefficient}
   \langle c_1(\mathfrak{s}),[\Tilde{F}] \rangle =\begin{cases} 
      p(rot_\mathbb{Q}(L')+1) & \text{ if } x+ny>0\\
      -p(rot_\mathbb{Q}(L')+1) & \text{ if } x+ny<0 
      \end{cases}.
\end{equation}

The last step is to represent $\rotq(L')$ using $\rotq(L)$. To simplify the notation in calculation we let $\rotq(L)=r$ and $\tbq(L)=a=\frac{x}{y}$ where $y$ is the order of $L$; then the order of $L'$ is $p=|x+ny|$ as we discussed above. Again we express contact $+n$ surgery on $L$ as surgery along a link $L_1 = L, L_2, \ldots, L_n$ following the DGS algorithm (Theorem \ref{3.2}), and $L_0$ is a further push off of $L_n$ that correspond to $L'$ before performing contact surgery. 
Then \cite[Lemma 4.1]{DingLiWuContact+1onrationalsphere} indicates that the rational rotation number of $L'$ is given by 

\begin{equation} \label{eqn: rotation calculation}
    \rotq(L')= \rotq(L_0)- \biggl \langle \begin{pmatrix} \rotq(L_1)\\ \vdots\\ \rotq(L_n) \end{pmatrix}, M^{-1} \begin{pmatrix} lk_\mathbb{Q}(L_0,L_1)\\ \vdots\\ lk_\mathbb{Q}(L_0,L_n) \end{pmatrix} \biggl \rangle.
\end{equation}

It's easy to find those rotation numbers. In $(Y,\xi)$ we have $L_1=L$ and $L_i=L^-$ for $i=0,2,...,n$, so $\rotq(L_1)=r$ and $\rotq(L_i)=r-1$ for $i=0,2,...,n$. In \eqref{eqn: rotation calculation}, 
$M$ is the $n\times n$ rational linking matrix \begin{equation*}
M = \begin{bmatrix} 
a+1 &a    &\dots   &\dots  &\dots &a \\ 
a   &a-2  &a-1 &\dots  &\dots &a-1 \\
\vdots &a-1 & a-2 &a-1 &\dots &\vdots \\ 
\vdots &\vdots & a-1 &a-2 &\dots &\vdots \\

\vdots &\vdots & \vdots &\vdots &\ddots &a-1 \\

a  &a-1  & \dots  &\dots &a-1      & a-2         \\
\end{bmatrix}
\end{equation*} whose inverse is 

\begin{equation*}
M^{-1} = \frac{1}{a+n} \begin{bmatrix} 
n-(n-1)a &a    &\dots   &\dots  &\dots &a \\ 
a   &1-a-n  &1 &\dots  &\dots &1\\
\vdots &1 & 1-a-n &1 &\dots &\vdots \\ 
\vdots &\vdots & 1 &1-a-n &\dots &\vdots \\

\vdots &\vdots & \vdots &\dots &\ddots &1 \\

a  &1  & \dots  &\dots &1      & 1-a-n         \\
\end{bmatrix}
\end{equation*}

Moreover since $lk_\mathbb{Q}(L_0,L_1)=a$, and $lk_\mathbb{Q}(L_0,L_i)=a-1$ for all $i=2,...n$ we want to calculate 
$$
    \begin{bmatrix} 
n-(n-1)a &a    &\dots   &\dots  &\dots &a \\ 
a   &1-a-n  &1 &\dots  &\dots &1\\
\vdots &1 & 1-a-n &1 &\dots &\vdots \\ 
\vdots &\vdots & 1 &1-a-n &\dots &\vdots \\

\vdots &\vdots & \vdots &\dots &\ddots &1 \\

a  &1  & \dots  &\dots &1      & 1-a-n         \\
\end{bmatrix} \begin{pmatrix} a\\a-1\\ \vdots\\ \vdots \\ \vdots \\ a-1 \end{pmatrix}
=\begin{pmatrix} a\\1\\ \vdots\\ \vdots \\ \vdots \\ 1 \end{pmatrix}$$

Thus equation \ref{eqn: rotation calculation} becomes 
\begin{align*}
    \rotq(L')&= (r-1) - \biggl \langle \begin{pmatrix} r\\ r-1\\\vdots\\ r-1 \end{pmatrix}, \frac{1}{a+n} \begin{pmatrix} a\\1\\ \vdots\\ 1\end{pmatrix} \biggl \rangle \\
    &=(r-1)-\frac{ra+(n-1)(r-1)}{a+n} \\
    &=\frac{r-a-1}{a+n} \text{ (substitute }a=\frac{x}{y}) \\
    &= \frac{yr-x-y}{x+ny}\\
    &= \frac{yr+ny-y-(x+ny)}{x+ny}\\
    &=\frac{y(r+n-1)}{x+ny}-1
\end{align*}

Now when we plug in back to equation \ref{sign of c1 depend on smooth coefficient} and  $p=|x+ny|$ we obtain 

\begin{align*}
    \langle c_1(\mathfrak{s}),[\Tilde{V}] \rangle & = \begin{cases} 
      p(rot_\mathbb{Q}(L')+1) & \text{ if } x+ny>0\\
      -p(rot_\mathbb{Q}(L')+1) & \text{ if } x+ny<0
      \end{cases}\\
    &=\begin{cases} 
      |x+ny|(\frac{y(r+n-1)}{x+ny}-1+1) & \text{ if } x+ny>0\\
      -|x+ny|(\frac{y(r+n-1)}{x+ny}-1+1) & \text{ if } x+ny<0
      \end{cases}\\
    &=  \begin{cases} 
      y(r+n-1) & \text{ if } x+ny>0\\
      y(r+n-1) & \text{ if } x+ny<0 \end{cases}\\
     &= y(r+n-1) \text{ \ \ \ if } x+ny \neq 0  
   .
\end{align*}

Last the condition both $Y$ and $Y_n(L)$ are rational homology sphere implies the order of $L'$, namely $p=|x+ny|$, is nonzero, which conclude the proof.
\end{proof}

\clearpage \bibliographystyle{plain}
\bibliography{bib}

\end{document}